\documentclass[12pt]{amsart}

\input{melaniemod.sty}

\usepackage{fullpage,pdfsync,caption, placeins, lscape}

\newtheorem{nts}{Note to self}

\title{Computing cohomology of configuration spaces}

\author{Megan Maguire\\
(with an appendix by Matthew Christie and Derek Francour) }
\address{Department of Mathematics\\
University of Wisconsin-Madison \\ 480 Lincoln Drive \\
Madison, WI 53705 USA} 
\email{mmaguire2@math.wisc.edu}

\begin{document}

\maketitle

\begin{abstract}
We give a concrete method to explicitly compute the rational cohomology of the unordered configuration spaces of connected, oriented, closed, even-dimensional manifolds of finite type which we have implemented in Sage \cite{sage}. As an application, we give a complete computation of the stable and unstable rational cohomology of unordered configuration spaces in some cases, including that of $\C\P^3$ and a genus 1 Riemann surface, which is equivalently the homology of the elliptic braid group. In an appendix, we also give large tables of unstable and stable Betti numbers of unordered configuration spaces. From these, we empirically observe stability phenomenon in the unstable cohomology of unordered configuration spaces of some manifolds, some of which we prove and some of which we state as conjectures. 

\end{abstract}

\section{Introduction}

\subsection{Cohomological stability of configuration spaces.}

Given a sequence of topological spaces or groups, $\{X_n\}$, \emph{(rational) cohomological stability} is the property that, for each $i\geq 0$, $$H^i(X_n; \Q) = H^i(X_{n+1}; \Q)$$ for $n \geq f(i)$, where $f(i)$ is some function of $i$. We call $H^i(X_n; \Q)$ for $n \geq f(i)$ the \emph{stable cohomology groups} of $\{X_n\}$ and $n \geq f(i)$ the \emph{stable range}. Conversely, $H^i(X_n; \Q)$ for $n < f(i)$ are the \emph{unstable cohomology groups} of $\{X_n\}$, and $n < f(i)$ is the \emph{unstable range}.

Let $X$ be a topological space and $n \in \N.$ The $n$th \emph{ordered configuration space} of $X$ (or the space of $n$ distinct labeled points in $X$) is $$\PConf^nX = \left(X^n - \bigcup_{1 \leq i < j \leq n} \Delta_{ij}\right)$$ where $\Delta_{ij} := \{(x_1, \hdots, x_n) \in X^n : x_i = x_j\}.$ The symmetric group $S_n$ acts on $\PConf^nX$ by permuting coordinates, and if we mod out by this action, we obtain the $n$th \emph{unordered configuration space} of $X$ (or the space of $n$ distinct unlabeled points in $X$) $$\Conf^nX := \PConf^nX / S_n.$$ In this paper, we are primarily interested in the unorderd configuration spaces, $\Conf^nX$, for $X$ a manifold.

This paper, using previously developed theoretical framework, gives a very concrete method to explicitly compute $H^i(\Conf^nX; \Q)$ for $X$ a connected, oriented, closed, even-dimensional manifold of finite type which we have implemented in Sage \cite{sage}. In the appendix, we give several tables of output from our program from which we empirically observe many more phenomena in the cohomology of configuration spaces than just the well-known phenomenon of cohomological stability.

There has been recent work showing that the actual stable Betti numbers of some specific manifolds, for example closed surfaces, have structure and exhibiting structure in the unstable Betti numbers, as well \cite{Drummond-ColeKnudsen2016, Scheissl2016, MillerWilson2016}. As an application of our approach, we give examples of such structure, some conjectured and some proven.

\subsection{Previous Work}
One of the first stability results for configuration spaces was given by Arnol'd \cite{Arnol'd1969} in 1969 who proved that there are inclusions $$\Conf^n\R^2 \hookrightarrow \Conf^{n+1}\R^2$$ that induce isomorphsims $$H_i(\Conf^n\R^2;\Z) \rightarrow H_i(\Conf^{n+1}\R^2; \Z)$$ for $n$ sufficiently large (depending on $i$).

McDuff \cite[Theorem 1.3]{McDuff1975} and Segal \cite[Proposition A.1]{Segal1979} generalized Arnol'd's result to open manifolds. McDuff defined a map (often referred to as the \emph{scanning map}) from $\Conf^nX$ to the the space of degree-k compactly-supported sections of the
fiberwise one-point compactification of the tangent bundle of $X$ and used that to prove integral homological stability for open manifolds. Segal then gave explicit bounds for the stable range and a new proof of homological stability via an argument more similar to Arnol'd's. 

Integral homological stability is known not to hold in general for closed manifolds. For example, $H_1(\Conf^nS^2; \Z) = \Z/(2n-2)\Z$ which has a dependence on $n$ that cannot be avoided by simply taking $n$ to be very large. For many years, it was thought that integral homological stability for open manifolds was the best one could do until Church considered the problem via the lens of \emph{representation stability}. By regarding $H^i(\PConf^n X; \Q)$ as a rational representation of $S_n$, Church \cite[Theorem 1]{Church2012} proved that the decomposition of $H^i(\PConf^nX; \Q)$ into irreducible representations remains the same in some sense for $n$ sufficiently large (depending on $i$) when $X$ is a connected, orientable, manifold with the homotopy type of a finite CW complex. (For example, the trivial representation of $S_n$ can be thought of as the same as the trivial representation of $S_m$ even when $n \neq m.$) Then as a corollary (\cite[Corallary 3]{Church2012}), Church concludes that $H^i(\Conf^n X; \Q) = H^i(\Conf^{n+1}X; \Q)$ for $n > i$, i.e. that the \emph{rational} cohomology groups of $\Conf^n X$ satisfy cohomological stability. Shortly following Church's proof, Randal-Williams \cite[Theorem C]{Randal-Williams2013} recovered Church's stability result using more traditional topological methods and was able to give an improvement on the bound of the stable range in some cases. Using the method of factorization homology, Knudsen \cite[Theorem 1.3]{Knudsen2015} was able to generalize this to non-orientable manifolds.

Combining these results yields the following stability theorem.

\begin{theorem}[\cite{Arnol'd1969, McDuff1975, Segal1979, Church2012, Randal-Williams2013, Knudsen2015}]
\label{stability}
For $X$ a connected manifold with $H^*(X;\Q)$ finite-dimensional and $n\geq i+1$, we have that
$$
H^i(\Conf^n X; \Q)=H^i(\Conf^{n+1} X;\Q).
$$
\end{theorem}
\subsection{Stable and unstable Betti numbers of configuration spaces}

Theorem \ref{stability} gives an important characterization of the rational cohomology of configuration spaces of manifolds, but a more fundamental question is ``given a manifold $X$, what are the Betti numbers of $\Conf^n(X)$?'' While there are several theoretical tools one can use to explicitly compute these Betti numbers (see \cite{CohenTaylor1978Springer, Totaro1996, FultonMacPherson1994, Kriz1994, FelixTanre2005, FelixThomas2000, FelixThomas2004, Knudsen2015, McDuff1975, BenderskyGitler1991, FadellNeuwirth1962}), surprisingly few have been computed. Using the theoretical framework of the Cohen--Taylor--Totaro--Kriz spectral sequence, we implemented an algorithm in Sage \cite{sage} to compute Betti numbers of unordered configuration spaces of connected, oriented, even-dimensional, closed manifolds of finite type (see the appendix for tables of our computations). As examples of our approach, in Section 4, we compute all the Betti numbers of $\Conf^n\C\P^1,$ $\Conf^n\C\P^2,$ and $\Conf^n\C\P^3$. Most of these were previously known, with the exception of the unstable Betti numbers of $\Conf^n\C\P^3.$ See Section 4 for remarks on methods used in previous work.

One of the first phenomena we empirically observed was the structure of the stable Betti numbers of the unordered configuration space of the closed genus 1 Riemann surface, $\Sigma_1$. At the time, very few Betti numbers of configuration spaces had been computed, and we were uncertain as to whether or not we would observe any structure in these values. We remark that the cohomology of $\Conf^n \Sigma_1$, is equivalently the cohomology of the elliptic braid group, $B_n(\Sigma_1)$ \cite{Birman1969, Scott1970}.

\begin{proposition}
\label{stableg1betti}
The stable Betti numbers of $\Conf^n\Sigma_1$ are $$b_0 = 1, b_1 = 2, b_2 = 3, b_3 = 5, b_4 = 7, \hdots, b_i = 2i-1, \hdots.$$
\end{proposition}

We give the full rational cohomology of $\Conf^n\Sigma_1$, i.e. also the unstable Betti numbers, in Section 4 (see Proposition \ref{P:g1SV}). Prior to our work, Napolitano \cite[Table 2]{Napolitano2003} had computed $H_i(\Conf^n\Sigma_1;\Z)$ for $1 \leq n \leq 6$ and $0 \leq i \leq 7$, and Kallel \cite[Corollary 1.7]{Kallel2008} computed $H_1(\Conf^n\Sigma_1; \Z)$ for $n \geq 3.$ Concurrently with our work, Scheissl \cite{Scheissl2016} independently computed the stable and unstable Betti numbers of $\Conf^n\Sigma_1$ also using the Cohen--Taylor--Totaro--Kriz spectral sequence, and Drummond-Cole and Knudsen \cite[Corollaries 4.5--4.7]{Drummond-ColeKnudsen2016} not only computed the stable and unstable Betti numbers of $\Conf^n\Sigma_1$, but did so for all surfaces of finite type via a method derived from factorization homology. For other related computations, see \cite{BrownWhite1981, BodigheimerCohen1988, BodigheimerCohenTaylor1989, Knudsen2015, Azam2015}.

In the appendix, we provide tables of Betti numbers we've computed for unordered configuration spaces of the following manifolds: $\C\P^1 \times \C\P^1$, $\C\P^1 \times \C\P^1 \times \C\P^1$, $\C\P^1 \times \C\P^2$, $\P_{\C\P^2}(\O \oplus \O(1))$, $\C\P^3$, $\C\P^4$, $\C\P^5$, $\C\P^6$, $\Sigma_1$, $\Sigma_1 \times \C\P^1$, $\Sigma_2$, $\Sigma_3$, and $\Sigma_4$ (where $\Sigma_g$ is the clsoed genus $g$ Riemann surface). For $X = \C\P^1 \times \C\P^2$ and $Y = \P_{\C\P^2}(\O\oplus \O(1))$, Totaro \cite[Section 5]{Totaro1996} pointed out that $\Conf^3X$ and $\Conf^3Y$ have different rational cohomology. In \cite[Section 1.19]{VakilWood2015}, Vakil and Wood ask if this difference goes away in the stable range. However, from tables \ref{Tab:cp1xcp2one} and \ref{Tab:cp1xcp2Tone}, we see that $\Conf^nX$ and $\Conf^nY$ do have significantly different rational cohomology; specifically, in the stable range, $H^{11}(\Conf^{15}X; \Q) \neq H^{11}(\Conf^{15}Y; \Q)$ and $H^{12}(\Conf^{15}X; \Q) \neq H^{12}(\Conf^{15}Y; \Q).$

For other computational results not discussed above, the interested reader should consult \cite{Cohen1973, CohenTaylor1978Springer, FultonMacPherson1994, BodigheimerCohenTaylor1989, FelixThomas2000, Napolitano2003, Azam2015, Knudsen2015, Fuks1970, BrownWhite1981, CohenTaylor1978BullAmer, Vainshtein1978, Kallel2008, Vershinin1999, FeichtnerZiegler2000, DominguezGonzalezLandweber2013, Sohail2010, BodigheimerCohen1988}.

\subsection{Stable instability and vanishing}
Traditional stability is the statement that $H^i(X_n;\Q) \approx H^i(X_{n+1};\Q)$ for $i$ fixed and $n$ sufficiently large. But one might reasonably ask about other relationships between the cohomology groups, i.e. stability where neither $i$ nor $n$ is fixed, but grow in some other way. For example, Miller and Wilson \cite[Theorem 1.2]{MillerWilson2016} have recently proven stability phenomenon following from maps the form $H^i(\PConf^nX;\Q) \rightarrow H^{i+1}(\PConf^{n+2}X;\Q)$ for some $X$. 

Additionally, Church, Farb, and Putman \cite{ChurchFarbPutman2014} have made conjectures about the unstable cohomology of $\{X_n\} = \{\SL_n(\Z)\}, \{\Mod_n\}$, where $\Mod_n$ is the mapping glass group of the genus $n$ closed Riemann surface. Specifically, if we let $D_n$ denote the virtual cohomological dimension of $X_n$, they conjecture:

\begin{itemize}
\item (Stable instability Conjecture) For $j\geq 0$, the group $H^{D_n-j}(X_n;\Q)$ does not depend on $n$ for $n$ sufficiently large, and
\item  (Vanishing Conjecture) For each $i\geq 0$, we have $H^{D_n-j}(X_n;\Q)=0$ for $n$ sufficiently large.
\end{itemize}

\begin{figure}[h]
\includegraphics[width=4.5in]{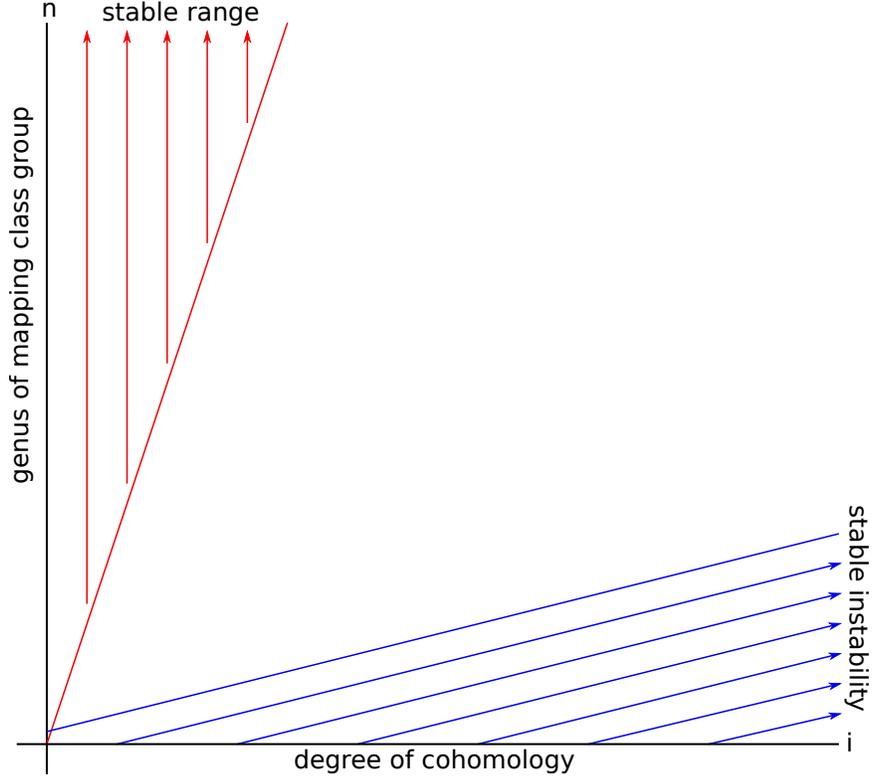}
\caption{Betti numbers of the mapping class group of genus $n$; $\dim H^i(\Mod_n;\Q)$ sits at the point $(i,n).$ Red lines illustrate traditional cohomological stability. Blue lines illustrate stable instability.}
\label{fig:mod_n}
\end{figure}

In particular, they prove that in these two cases, the stable instability conjecture implies the vanishing conjecture. 

We observe similar phenomena occuring in the cohomology of unordered configuration spaces of manifolds. As an example, our approach gives a simple proof of the following theorem which follows from a result first proved by Kallel \cite[Theorem 1.1]{Kallel2008} and was also proven by Napolitano \cite[Theorem 3]{Napolitano2003} in the case when $X$ is any connected surface.

\begin{theorem}[Stable Instability and  Vanishing Theorem, \cite{Kallel2008}]\label{T:SIVC}
Let $X$ be an oriented real manifold of dimension $D$. Then for $n\geq j+2$, we have
$$
H^{nD-j}(\Conf^n X;\Q)=0.
$$
\end{theorem}
Said another way, for $i\geq (D-1)n+2$, we have
$$
H^{i}(\Conf^n X;\Q)=0.
$$
For example, when $X$ is an oriented, compact surface, we see that the only possibly non-zero unstable cohomology is $H^i(\Conf^n X;\Q)$ for $i=n,n+1$.

In general, the vanishing in Theorem~\ref{T:SIVC} is sharp, as demonstrated by the following proposition.

\begin{proposition}\label{P:gCD}
Let $\Sigma_g$ be a closed Riemann surface of genus $g$ with $g\geq 1$.  Then for $n\geq 3$,
$$
H^{n+1}(\Conf^n \Sigma_g; \Q)\ne 0.
$$
\end{proposition}

In  many cases, for example in the case of $\Sigma_1$, the only alternate stability that we observe stabilizes to zero. However, in some examples, we do see some \emph{non-vanishing stable instability}. Specifically, for some values of $i$ (depending on $n$), we have that $H^i(\Conf^n\C\P^3;\Q) \approx H^{i+2}(\Conf^{n+1}\C\P^3;\Q) \neq 0$ for $n$ sufficiently large. (See Figure \ref{fig:cp3} for an illustration of stable instability.)

\begin{proposition}[Stable instability for $\C\P^3$]\label{cp3stableinstability} For $n \geq \frac{i}{2},$ $$H^i(\Conf^n \C\P^3;\Q) = H^i(\Conf^{n+1}\C\P^3; \Q);$$ i.e. the stable range of $\Conf^n \C\P^3$ is $n \geq \frac{i}{2}.$ Furthermore, for all $n \geq 11$ and $i > 2n$,  $$H^{i}(\Conf^n\C\P^3; \Q)  = H^{i+2}(\Conf^{n+1}\C\P^3; \Q).$$ Specifically, $$H^{2n+j}(\Conf^n\C\P^3; \Q) = H^{2(n+1)+j}(\Conf^{n+1}\C\P^3;\Q) = \Q$$ for $n \geq 11$ and $j \in \{2,4,6,7,8,10,12\}$, $$H^{2n+j}(\Conf^n\C\P^3; \Q) = H^{2(n+1)+j}(\Conf^{n+1}\C\P^3;\Q) = \Q^2$$ for $n \geq 10$ and $j \in \{1,3,5\},$
and $$H^{2n+j}(\Conf^n\C\P^3; \Q) = H^{2(n+1)+j}(\Conf^{n+1}\C\P^3;\Q) = 0$$ for $n \geq 1$ and $j \in \N - \{1, \hdots, 8,10,12\}.$
\end{proposition}

Our bound for the stable range of $\Conf^n \C\P^3$ improves slightly on the bound of $n \geq \frac{i}{2} + 1$ given by Church \cite[Proposition 4.1]{Church2012}.

Our computations (see tables \ref{Tab:cp3one}--\ref{Tab:cp6five} and propositions \ref{cp1betti}--\ref{cp3betti}) have led us to make the following conjecture about the top non-vanishing cohomology of $\Conf^n\C\P^k$ for all $k \geq 1$, which we also believe to be an example of non-vanishing stable instability.

\begin{conjecture}[Non-vanishing stable instability for $\C\P^k$] \label{stableinstability}
Given a positive integer $k$, there exist positive integers $n_0$ and $c_s$ such that, for $n \geq \frac{i-c_s}{2\lceil k/2\rceil -2},$ $$H^i(\Conf^n\C\P^k;\Q) = H^i(\Conf^{n+1}\C\P^k;\Q),$$  and for all $n \geq n_0$ and $i > (2\lceil k/2\rceil -2)n + c_s,$ $$H^i(\Conf^n\C\P^k;\Q) \approx H^{i + 2\lceil k/2 \rceil -2}(\Conf^{n+1}\C\P^k; \Q).$$ Specifically, there exists $c_t$ such that $$H^{(2\lceil k/2 \rceil -2)n + c_t}(\Conf^n\C\P^k;\Q) = H^{(2\lceil k/2 \rceil -2)(n+1) + c_t}(\Conf^{n+1}\C\P^k;\Q) = \Q$$ for all $n \geq n_0$, and $$H^{(2\lceil k/2 \rceil -2)n + j}(\Conf^n\C\P^k;\Q) = H^{(2\lceil k/2 \rceil -2)(n+1) + j}(\Conf^{n+1}\C\P^k;\Q) = 0$$ for all $n \geq n_0$ and $j > c_t.$
\end{conjecture}

\begin{figure}[h]
\includegraphics[width=4.5in]{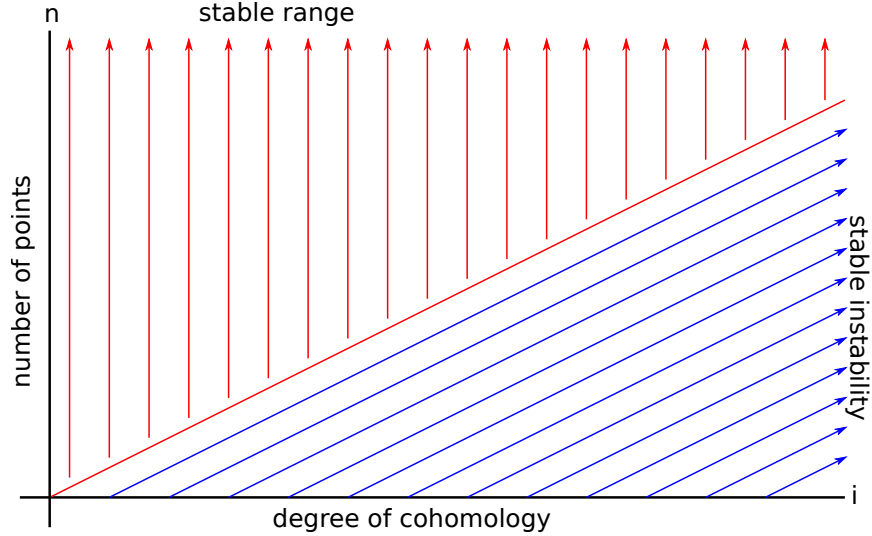}
\caption{Betti numbers of the unordered configuration space of $n$ points in $\C\P^3$; $\dim H^i(\Conf^n\C\P^3)$ sits at the point $(i,n).$ Red lines illustrate traditional cohomological stability. Blue lines illustrate stable instability.}
\label{fig:cp3}
\end{figure}

For $k = 3,$ Conjecture \ref{stableinstability} follows from Proposition \ref{cp3stableinstability}. In Section 4, we will see that Conjecture~\ref{stableinstability} is also true for $k=1,2$, and in particular that $H^3(\Conf^n\C\P^1;\Q) = \Q$ is the top non-vanishing cohomology of $\Conf^n\C\P^1$ and that $H^{11}(\Conf^n\C\P^2; \Q) = \Q$ is the top non-vanishing cohomology of $\Conf^n\C\P^2.$ 


\section{The Cohen--Taylor--Totaro--Kriz spectral sequence}

There are several theoretical tools that can be used to compute the cohomology of configuration spaces (for example, see \cite{CohenTaylor1978Springer, Totaro1996, FultonMacPherson1994, Kriz1994, FelixTanre2005, FelixThomas2000, FelixThomas2004, Knudsen2015, McDuff1975, BenderskyGitler1991, FadellNeuwirth1962}). But the one we use is a Leray spectral sequence for the fibration $\PConf^nX \hookrightarrow X^n$ converging to $H^*(\PConf^nX;\Q)$ whose $S_n$-invariants converge to $H^*(\Conf^nX;\Q)$ when $X$ is a connected, oriented manifold of finite type.

Cohen and Taylor \cite[Section 2]{CohenTaylor1978Springer} were the first to describe this spectral sequence (although they failed to identify it as Leray). They described the $E_2$-page and the first nontrivial differential for $X$ a connected, oriented manifold of finite type. Then Totaro \cite[Theorem 3]{Totaro1996} and Kriz \cite[Theorem 1.1]{Kriz1994} proved that the sequence collapses after the first nontrivial differential for $X$ a smooth, complex projective variety. F\'{e}lix and Thomas \cite[Theorem A]{FelixThomas2000} and Church \cite[Proposition 4.3]{Church2012v1} described the $S_n$-invariants of the $E_2$-page for $X$ even-dimensional. In the odd-dimensional case, Bodigheimer, Cohen, and Taylor \cite[Theorem C]{BodigheimerCohenTaylor1989} and F\'{e}lix and Tanr\'{e} \cite[Theorem 4]{FelixTanre2005} proved that $H^*(\Conf^nX; \Q) = \bigoplus_{i+j = n} \Sym^iH^{even}(X;\Q) \otimes \bigwedge^j H^{odd}(X;\Q).$ F\'{e}lix and Tanr\'{e} \cite[Theorem 3]{FelixTanre2005} and Knudsen\cite[Corollary 4.11]{Knudsen2015} proved that, in the even-dimensional case, the Cohen--Taylor--Totaro--Kriz spectral sequence only has one nontrivial differential for $X$ a connected manifold of finite type. 

We now give a description of what this spectral sequence looks like. Let $X$ be an even-dimensional manifold.  Let $V_X$ be the bi-graded vector space $\tilde{H}^*(X,\Q)$, with second grading identically $0$.  Let $W_X$ be the bi-graded vector space
$H^*(X;\Q)$ in which an element of degree $g$ from $\tilde{H}^*(X,\Q)$ has grade $(g,1)$.
Let $R_X=\Q[V_X\oplus W_X]$ be the free graded-commutative algebra on $V_X\oplus W_X$, in which the grading for the graded-commutativity is given by the sum of the two grades in the bigrading of $V_X\oplus W_X$.  We put a third grading, \emph{length}, on $R_X$ induced by giving elements of $V_X$ length $1$ and $W_X$ length $2$.

\begin{theorem}[\cite{CohenTaylor1978Springer, FelixThomas2000, Church2012v1}]\label{T:E2}
Let $X$ be an oriented even dimensional manifold.  Let $\mathcal{E}_2(n)$ be the $S_n$ invariants of the $E_2$ page of the Leray spectral sequence for $\PConf^nX \hookrightarrow X^n$, with the rows that are not indexed by multiples of $\dim X-1$ removed.  Then we have a isomorphism of bi-graded vector spaces
$$
\mathcal{E}_2(n) \isom R_{X,\leq n},
$$
where $R_{X,\leq n}$ is the quotient of $R_{X}$ be elements of length $>n$.  
\end{theorem}

We differ from the standard literature by giving the elements of $W_X$ the grade $(\ast, 1)$ instead of $(\ast, \dim X -1)$. We make this choice so as to remove all the zero rows, i.e. those not indexed by multiples of $\dim X -1$, from the Cohen--Taylor--Totaro--Kriz spectral sequence. (Note that the parity of the grading is preserved since we are restricting to the case when $\dim X$ is even.) So, if $R_{X, \leq n}^{p,q}$ denotes the elements of $R_{X, \leq n}$ with bi-grade $(p,q)$, Theorem \ref{T:E2} gives a vector space isomorphism $\mathcal{E}_2^{p,q}(n) \rightarrow R_{X, \leq n}^{p,q}.$

In order to describe how to interpret the differential on $\mathcal{E}_2(n)$ in the $R_{X, \leq n}$ setting, we first need to develop some notation for the elements of $R_{X, \leq n}.$ Let $\{y_0 = 1, y_1, \hdots, y_m\}$ be a graded basis of $H^*(X;\Q).$ Then we denote the corresponding basis of $V_X$ by $\{Y_1, \hdots, Y_m\}$, where $y_i$ corresponds to $Y_i$, and we denote the corresponding basis of $W_X$ by $\{\Y_0 = \1, \Y_1, \hdots, \Y_m\}$, where $y_i$ corresponds to $\Y_i$. This induces bases on $R_X$ and $R_{X, \leq n}.$ Namely, we see that $$B^{p,q}(n) = \left\{\begin{array}{c|c}Y^{\r}\Y^{\s} & \ell(Y^{\r}\Y^{\s}) \leq n, \sum_{i=1}^m (r_s + s_i)|y_i| = p, \sum_{i=0}^m s_i = q, \mbox{~and~}r_i, s_j \in \{0,1\}\\ ~ & \mbox{~for~} |y_i| \mbox{~odd,~} |y_j| \mbox{~even}\end{array}\right\}$$ is a basis for $R_{X, \leq n}^{p,q},$ where $Y^{\r}\Y^{\s} : = Y_1^{r_1}\cdots Y_m^{r_m}\Y_0^{s_0}\cdots \Y_m^{s_m}.$

\begin{theorem} [\cite{CohenTaylor1978Springer,FelixTanre2005}]
Let $X$ be a connected, oriented, closed, even-dimensional manifold of finite type. Under the isomorphism of Theorem~\ref{T:E2}, the differential on $\mathcal{E}_2(n)$ corresponds to the differential $d$ on $R_X$ such that $d(V_X)=0$ and, for $h \in H^*(X;\Q)$ with corresponding element $\H \in W_X$, we have
$$
d \H = \frac{1}{2}\sum_{i=0}^m (-1)^{|y_i^{\vee}|}\H_i \Y_i^{\vee}
$$

where $\{y_0^{\vee}, \hdots, y_m^{\vee}\}$ is a dual basis of $\{y_0, \hdots, y_m\}$ with respect to the cup product pairing (i.e. $y_i \cdot y_i^{\vee} = y_m$ where $H^{\dim X}(X;\Q) = \Q \cdot y_m$), $\H_i$ is the element of $W_X$ corresponding to $h \cdot y_i$, and $\Y_i^{\vee}$ is the element of $W_X$ corresponding to $y_i^{\vee}.$
\end{theorem}

Note that due to our convention of giving elements of $W_X$ the grade $(\ast, 1)$, our differentials now have different codomains, i.e. $d:E_2^{p,q}(n) \rightarrow E_2^{p+D, q-1}(n).$

\begin{theorem}[\cite{Totaro1996, Kriz1994, FelixTanre2005, Knudsen2015}]\label{T:1D}
The spectral sequence $\mathcal{E}_2(n)$ has only one nontrivial differential.
\end{theorem}

Combining theorems \ref{T:E2} -- \ref{T:1D}, turns the task of computing the cohomology groups $H^i(\Conf^n X; \Q)$ into a linear algebra problem that one can program a computer to solve. 

\subsection*{Notation} Throughout the rest of this paper, we will refer to $R_{X, \leq n}$ as $E_2(n),$ and $E_2^{p,q}(n)$ will refer to the elements of $R_{X, \leq n}$ with bi-grade $(p,q).$ We will sometimes just say $E_2$, the $E_2$-page, or $E_2^{p,q}$ if $n$ is fixed. For clarity, we shall sometimes write $d^{p,q}$ for $d: E_2^{p,q}(n) \rightarrow E_2^{p+D,q-1}(n).$

\section{An exact subcomplex}

We now define a subcomplex of the cochain complex associated to the $E_2$-page involving the orientation class of $X$. Using a filtration on the cochain complex, we show that this subcomplex is exact. This allows us to quotient out by it and is the key fact that allows us to prove Theorem \ref{T:SIVC}. The ordered version of our subcomplex is defined in \cite[Proposition 5.5]{AshrafAzamBerceanu2014}, although it takes a bit of work to see this.

\begin{proposition}
\label{top}
Let $X$ be a complex manifold of real dimension $D$, let $B(X) = \{y_0=1, y_1, \hdots, y_m\}$ be a graded basis of $H^*(X; \Q),$ and fix $n \geq 1.$ Let $C^i(n) = \bigoplus_{p+(D-1)q = i} E_2^{p,q}(n)$ be the cochain complex associated to the $E_2$-page of $\Conf^nX.$ Recall that $$B^{p,q}(n) = \left\{\begin{array}{c|c}Y^{\r}\Y^{\s} & \ell(Y^{\r}\Y^{\s}) \leq n, \sum_{i=1}^m (r_s + s_i)|y_i| = p, \sum_{i=0}^m s_i = q, \mbox{~and~}r_i, s_j \in \{0,1\}\\ ~ & \mbox{~for~} |y_i| \mbox{~odd,~} |y_j| \mbox{~even}\end{array}\right\}$$ is a basis for $E_2^{p,q}(n).$ Let $A^{p,q}(n) \subseteq E_2^{p,q}(n)$ denote the subspace $$A^{p,q}(n) := \vspan\{Y^{\r}\Y^{\s} \in B^{p,q}(n) | r_m \geq 2 \mbox{~or~} s_m \geq 1\},$$ and let $\mathbf{A}(n)$ be the corresponding subcomplex of $\mathbf{C}(n)$. We claim that $\mathbf{A}(n)$ is exact.
\end{proposition}

\begin{proof}
To prove the claim, we shall impose a filtration $\{\mathbf{A_k}(n)\}_{k \in \N}$ on $\mathbf{A}(n)$ and show, via induction, that each $\mathbf{A_k}(n)$ is an exact subcomplex of $\mathbf{C}(n).$ To this end, define $A_k^{p,q}(n) \subseteq E_2^{p,q}(n)$ to be the subspace spanned by $$B_k^{p,q}(n) : = \left \{ \begin{array}{c|c}Y^{\r}\Y^{\s} \in B^{p,q}(n) & \sum_{i=0}^{m-1} s_i \leq k, \mbox{~and~} r_m \geq 2 \mbox{~or~} s_m \geq 1 \end{array}\right\}.$$ To show that $\mathbf{A_k}(n)$ is indeed a subcomplex of $\mathbf{A}(n),$ consider our two ``typical" basis elements of $A_k^{p,q}(n)$: $Y^{\r}\Y_0^{s_0} \cdots \Y_{m-1}^{s_{m-1}}$ with $r_m \geq 2$ and  $Y^{\r}\Y_0^{s_0}\cdots \Y_{m-1}^{s_{m-1}}\Y_m^{s_m}$ with $s_m \geq 1.$ 
Now $$d(Y^{\r}\Y_0^{s_0}\Y_1^{s_1}\cdots \Y_{m-1}^{s_{m-1}}) = Y^{\r}\cdot d(\Y_0^{s_0} \cdots \Y_{m-1}^{s_{m-1}})$$ so $d(Y^{\r}\Y_0^{s_0}\cdots \Y_{m-1}^{s_{m-1}}) \in \mathbf{A}_k(n).$ Furthermore, $Y^{\r}\Y_0^{s_0}\cdots \Y_{m-1}^{s_{m-1}} \in A_k^{p,q}(n)$ means that $q \leq k,$ and since $d:A^{p,q}(n) \rightarrow A^{p+D, q-1}(n),$ we see that $d(Y^{\r}\Y_0^{s_0}\Y_1^{s_1}\cdots \Y_{m-1}^{s_{m-1}}) \in A_k^{p+D, q-1}(n).$ Similarly, for $Y^{\r}\Y_0^{s_0} \cdots \Y_{m-1}^{s_{m-1}}\Y_m^{s_m}$ with $s_m \geq 1$,
\begin{eqnarray*}
d(Y^{\r}\Y_0^{s_0}\cdots\Y_{m-1}^{s_{m-1}}\Y_m^{s_m}) & = & Y^{\r}d(\Y_0^{s_0}\cdots \Y_{m-1}^{s_{m-1}})\Y_m^{s_m} + (-1)^tY^{\r}\Y_0^{s_0}\cdots \Y_{m-1}^{s_{m-1}}\cdot d(\Y_m^{s_m})\\
& = & Y^{\r}\cdot d(\Y_0^{s_0}\cdots \Y_{m-1}^{s_{m-1}})\Y_m^{s_m}\\
&~& + (-1)^t s_m Y_1^{r_1} \cdots Y_{m-1}^{r_{m-1}}Y_m^{r_m+2}\Y_0^{s_0}\cdots \Y_{m-1}^{s_{m-1}}\Y_m^{s_m-1},
\end{eqnarray*}
where $t = \sum_{i=0}^{m-1}s_i(|y_i| + 1).$ So $d(Y^{\r}\Y_0^{s_0}\cdots\Y_{m-1}^{s_{m-1}}\Y_m^{s_m}) \in \mathbf{A_k}(n)$ when $s_m\geq 1.$ Since $Y_1^{r_1} \cdots Y_m^{r_m}\Y_0^{s_0}\cdots \Y_m^{s_m} \in A_k^{p,q}(n)$ implies that $q \leq k+s_m,$ and $d: A^{p,q}(n) \rightarrow A^{p+D, q-1}(n),$ we see that $d(Y_1^{r_1} \cdots Y_m^{r_m}\Y_0^{s_0}\cdots \Y_{m-1}^{s_{m-1}}\Y_m^{s_m}) \in A_k^{p+D, q-1}(n).$ Thus $\mathbf{A_k}(n)$ is indeed a subcomplex of $\mathbf{A}(n)$. 

We first prove that $\mathbf{A_0}(n)$ is exact. First consider $A_0^{p,0}(n)$. Since $A_0^{p+D,-1}(n) = \{0\},$ we see that $\ker(d^{p,0}) = A_0^{p,0}(n).$ Let $Y^{\r} \in B_0^{p,0}(n).$ Then $$d(Y_1^{r_1}\cdots Y_{m-1}^{r_{m-1}}Y_m^{r_m-2}\Y_m) = Y^{\r}$$ and $Y_1^{r_1}\cdots Y_{m-1}^{r_{m-1}}Y_m^{r_m-2}\Y_m \in A_0^{p-D,1}(n)$ since $s_m \geq 2.$ Thus $\ker\left(d^{p,0}|_{A_0^{p,0}}\right) = \im\left(d^{p-D,1}|_{A_0^{p-D,1}}\right).$

Now suppose $q \geq 1$, and let $$\alpha = \sum_{Y^{\r}\Y_m^q \in B_0^{p,q}(n)} c_{\r}Y^{\r}\Y_m^q \in \ker(d^{p,q})$$ where $c_{\r} \in \Q.$ Then $$d(\alpha) = \sum_{Y^{\r}\Y_m^q \in B_0^{p,q}(n)} c_{\r}qY_1^{r_1}\cdots Y_{m-1}^{r_{m-1}}Y_m^{r_m+2}\Y_m^{q-1} = 0.$$ But the only way $d(\alpha) = 0$ is if $\alpha = 0.$ Thus $\ker\left(d^{p,q}|_{A_0^{p,q}}\right) = \{0\}$ for $q\geq 1$. It follows trivially that $\ker\left(d^{p,q}|_{A_0^{p,q}}\right) = \im\left(d^{p-D, q+1}|_{A_0^{p-D,q+1}}\right)$, and so $\mathbf{A_0}(n)$ is exact.

As our inductive hypothesis, suppose that $\mathbf{A_{k-1}}(n)$ is exact. To show that $\mathbf{A_k}(n)$ is exact, it suffices to show that $\mathbf{A_k}(n)/\mathbf{ A_{k-1}}(n) =: \mathbf{\tilde{A}_k}(n)$ is exact. Note that $\tilde{A}_k^{p,q}(n) = \{0\}$ for $q < k.$ First suppose that $q = k.$ Then $$\tilde{B}_k^{p,k}(n) := B_k^{p,k}(n) - B_{k-1}^{p,k}(n) = \left\{Y^{\r}\Y_0^{s_0} \cdots \Y_{m-1}^{s_{m-1}} \in B_k^{p,k}(n) \right\}.$$ Since $\tilde{A}_k^{p+D, k-1}(n) = \{0\},$ we see that $\ker\left(d^{p,k}|_{\tilde{A}_k^{p,k}(n)}\right) = \tilde{A}_k^{p,k}(n).$ Given $Y^{\r}\Y_0^{s_0} \cdots \Y_{m-1}^{s_{m-1}} \in \tilde{A}_k^{p, k}(n),$ we see that $$d(Y_1^{r_1} \cdots Y_{m-1}^{r_{m-1}} Y_m^{r_m-2} \Y_0^{s_0} \cdots \Y_{m-1}^{s_{m-1}}\Y_m) = Y^{\r} \Y_0^{s_0} \cdots \Y_{m-1}^{s_{m-1}},$$ and $Y_1^{r_1} \cdots Y_{m-1}^{r_{m-1}}Y_m^{r_m-2} \Y_0^{s_0} \cdots \Y_{m-1}^{s_{m-1}}\Y_m \in \tilde{A}_k^{p-D, k+1}(n)$ since $s_m = 1.$ 

Now consider $\tilde{A}_k^{p,q}(n)$ for $q \geq k+1.$ Then $$\tilde{B}_k^{p, q}(n) = B_k^{p, q}(n) - B_{k-1}^{p, q}(n) = \left \{Y^{\r}\Y_0^{s_0} \cdots \Y_{m-1}^{s_{m-1}}\Y_m^{q-k} \in B_k^{p, q}(n) \right \}.$$ Let $$\alpha = \sum_{Y^{\r}\Y^{\s} \in \tilde{B}_k^{p, q}}(n) c_{\r, \s}Y^{\r}\Y_0^{s_0}\cdots \Y_{m-1}^{s_{m-1}}\Y_m^{q-k} \in \ker\left(d^{p,q}|_{\tilde{A}_k^{p, q}(n)}\right),$$ where $c_{\r,\s} \in \Q.$ Then 
\begin{eqnarray*}
d(\alpha) & = & \sum_{Y^{\r}\Y^{\s} \in \tilde{B}_k^{p, q}(n)} c_{\r, \s}Y^{\r}\cdot d(\Y_0^{s_0}\cdots \Y_{m-1}^{s_{m-1}}\Y_m^{q-k})\\
& = & \sum_{Y^{\r}\Y^{\s} \in \tilde{B}_k^{p, q}(n)} \left( c_{\r, \s}Y^{\r}\cdot d(\Y_0^{s_0} \cdots \Y_{m-1}^{s_{m-1}}) \cdot \Y_m^{q-k} + (-1)^t c_{\r, \s} Y^{\r} \Y_0^{s_0} \cdots \Y_{m-1}^{s_{m-1}} \cdot d(\Y_m^{q-k})\right)\\
& = & \sum_{Y^{\r}\Y^{\s} \in \tilde{B}_k^{p, q}(n)} (-1)^t(q-k) c_{\r, \s} Y_1^{r_1} \cdots Y_{m-1}^{r_{m-1}} Y_m^{r_m+2} \Y_0^{s_0} \cdots \Y_{m-1}^{s_{m-1}}\Y_m^{q-k-1}\\
& = & 0
\end{eqnarray*}
where $\displaystyle{ t = \sum_{i=0}^{m-1} s_i(|y_i|+1).}$
From this computation, we see that the only way $d(\alpha) = 0$ is if $\alpha = 0.$ Thus $\ker\left(d^{p,q}|_{\tilde{A}_k^{p,q}(n)}\right) = 0$, and so we trivially have that $\ker\left(d^{p,q}|_{\tilde{A}_k^{p,q}(n)}\right) = \im\left(d^{p-D,q+1}|_{\tilde{A}_k^{p-D,q+1}(n)}\right).$ Thus $\mathbf{A_k}(n)$ is exact, and by induction it follows that $\mathbf{A}(n)$ is exact.
\end{proof}

From now on, we will always work in the $E_2$-page quotiented out by this exact subcomplex (i.e. that $r_m \in \{0,1\}$ and $s_m = 0$), but by a slight abuse of notation, we will still refer to the $p,q$ entry as $E_2^{p,q}(n).$  

Armed with Proposition \ref{top}, we may now prove Theorem \ref{T:SIVC}.

\begin{proof}[Proof of Theorem \ref{T:SIVC}.]
Fix $q \geq 0.$ Recall that given $Y^{\mathbf{r}}\Y^{\mathbf{s}} \in B^{p,q}(n),$ 
\begin{eqnarray*}
\ell(Y^{\r}\Y^{\s}) & = & 2s_0 + \sum_{i=1}^m(r+i+2s_i)\\
& = & 2q + \sum_{i=1}^m r_i\\
& \leq & n.
\end{eqnarray*}
Thus $\sum_{i=1}^m r_i \leq n-2q.$ Additionally, from Proposition \ref{top}, we know that $r_m \in \{0,1\}$ and $s_m = 0.$

It follows that for a fixed $q$ and $n$, $Y^{\r}\Y^{\s}$ can have a $p$-value of at most $q(D-1) + (n-1-2q)(D-1) + D$, and so $B^{p,q}(n) = \emptyset$ for 
\begin{eqnarray}
 p & \geq & q(D-1) + (n-2q-1)(D-1) + D + 1\nonumber\\
 & = & (n-q)(D-1) + 2.\label{star}
\end{eqnarray}

Recall that $$H^i(\Conf^nX;\Q) = \bigoplus_{p+(D-1)q = i} E_\infty^{p,q}(n).$$ From \eqref{star}, we have that, for a fixed $n$, $$E_\infty^{p,q}(n) = \{0\}$$ for $p+(D-1)q \geq n(D-1)+2.$ Thus it follows that $H^i(\Conf^nX;\Q) = 0$ for $i \geq n(D-1) + 2.$
\end{proof}

The bound given by Theorem \ref{T:SIVC} is sharp. As we stated earlier in Proposition \ref{P:gCD},  $$H^{n+1}(\Conf^n\Sigma_g; \Q) \neq 0$$ for $n \geq 3.$ We now give the proof of Proposition \ref{P:gCD}.

\begin{proof}[Proof of Proposition \ref{P:gCD}.]
Recall that $\{1, a_1, \cdots, a_g, b_1, \cdots, b_g, t\}$ is a graded basis of $H^*(X;\Q)$ with $a_ib_i = -b_ia_i = t$ for all $1 \leq i \leq g.$ From our description of $d$ given in Section 2, we see that $$d(\A^{\r}\B^{\s}) = -2\sum_{i=1}^g(r_iA_iT\A^{\r-\mathbf{e_i}}\B^{\s} + s_iB_iT\A^{\r}\B^{\s - \mathbf{e_i}})$$ where $\mathbf{e_i}$ is the vector whose $i$th component is $1$ and all other components are $0$. 

First suppose that $n$ is even. Then $n = 2k$ for some $k \in \N.$ We shall show that $E_\infty^{k+2,k-1}(n) \neq \{0\}$, and thus $$H^{n+1}(\Conf^nX;\Q) = H^{2k+1}(\Conf^nX;\Q) = \bigoplus_{p+q = 2k+1} E_\infty^{p,q}(n) \neq \{0\}.$$

Now $E_2^{k+2,k-1}(n)$ has basis $$B^{k+2,k-1}(n) = \left\{A_iT\A^{\r}\B^{\s}, B_iT\A^{\r}\B^{\s} : 1 \leq i \leq g, \sum_{i=1}^g(r_i+s_i) = k-1\right\}.$$ Thus $$\dim E_2^{k+2,k-1}(n) = 2g \cdot \binom{k+2g-2}{2g-1}.$$ Since 
$d(A_iT\A^{\r}\B^{\s}) = d(B_iT\A^{\r}\B^{\s}) = 0$ for $1 \leq i \leq g$, it follows that $\rank d^{k+2, k-1} = 0.$

Now $E_2^{k,k}(n)$ has basis $$B^{k,k}(n) = \left\{\A^{\r}\B^{\s} : \sum_{i=1}^g(r_i+s_i) = k \right\}.$$ Then $$\rank d^{k,k} \leq \dim E_2^{k,k}(n) = \binom{k+2g-1}{2g-1}.$$

It follows that
\begin{eqnarray*}
\dim E_\infty^{k+2,k-1}(n) & = & \dim E_2^{k+2,k-1} - \rank d^{k+2,k-1} - \rank d^{k,k}\\
& \geq & 2g \cdot \binom{k+2g-2}{2g-1} - \binom{k+2g-1}{2g-1}\\
& = & \frac{2g \cdot (k+2g-2)!}{(k-1)!(2g-1)!} - \frac{(k+2g-1)!}{k! \cdot (2g-1)!}\\
& = & (k-1)\cdot \binom{k+2g-2}{2g-2}\\
& > & 0,
\end{eqnarray*}
since $g \geq 1$ and $k \geq 2.$

Now suppose that $n$ is odd. Then $n = 2k+1$ for some $k \in \N.$ In a manner similar to the above, we shall show that $E_{\infty}^{k+2,k} \neq \{0\}$, and it will thus follows that $$H^{n+1}(\Conf^nX;\Q) = H^{2k+2}(\Conf^nX;\Q) = \bigoplus_{p+q = 2k+2} E_\infty^{p,q}(n) = \{0\}.$$

We see that $E_2^{k+2,k}(n)$ has basis $$B^{k+2,k}(n) = \left\{T\A^{\r}\B^{\s} : \sum_{i=1}^g(r_i+s_i) = k \right\} \neq \emptyset.$$ Since $d(T\A^{\r}\B^{\s}) = 0$, we see that $\rank d^{k+2,k} = 0.$ 

Furthermore, $B^{k,k+1}(n) = \emptyset$, so $\rank d^{k,k+1} = 0.$ Thus $$\dim E_\infty^{k+2,k}(n) = \dim E_2^{k+2,k} - \rank d^{k+2,k} - \rank d^{k, k+1} > 0,$$ and so $H^{n+1}(\Conf^nX;\Q) \neq 0.$
\end{proof}

\section{Proofs of unstable and stable values}
In this section, we give concrete calculations for $H^*(\Conf^nX;\Q)$ for four spaces: $\C\P^1$, $\C\P^2$, $\C\P^3$, and $\Sigma_1$. This proves propositions \ref{stableg1betti} and \ref{cp3stableinstability}.

Recall that the $i$th Betti number $b_i(n)$ is given by $$b_i(n) = \sum_{p+(D-1)q = i} \dim(E_\infty^{p,q}(n)),$$ and $$\dim(E_\infty^{p,q}(n)) = \dim(E_2^{p,q}(n)) - \rank(d^{p,q}) - \rank(d^{p-D, q+1}).$$ Thus for the examples given below, we begin by computing bases for each $E_2^{p,q}(n)$, then we compute $d^{p,q}$ of each basis and determine the rank of each $d^{p,q}$. From there, we add and subtract the relevant quantities to determine the stable and unstable values of the Betti numbers.

\subsection{Case of $\C\P^1$}
Recall that $\{1, x\}$ is a graded basis of $H^*(\C\P^1;\Q)$ with $|1| = 0$ and $|x| = 2.$ An arbitrary basis element of the $E_2$-page is of the form $X^{r}\1^{s}$ with $r, s \in \{0,1\}.$ Thus we see that $B^{p,q}(n) = \emptyset$ for $p \geq 3$ and $q \geq 1.$ Since only a finite number of $B^{p,q}(n)$ are nonempty, our program can compute all the Betti numbers of $\Conf^n\C\P^1,$ proving the following proposition.

\begin{proposition}
\label{cp1betti} 
The stable Betti numbers of $\Conf^n \C\P^1$ are $b_0 = b_3 = 1$ and $b_i = 0$ for $i \neq 0,3.$ Furthermore, the entire cohomology is given by 
$$
H^i(\Conf^n\C\P^1; \Q) = \begin{cases}
\Q & \mbox{if } i = 0\\
\Q & \mbox{if } n = 1, i = 2\\
\Q & \mbox{if } n \geq 3, i =3\\
0 & \mbox{else}
\end{cases}.
$$
\end{proposition}

These numbers have been previously computed by many others. Cohen and Taylor \cite[Proposition 5.4]{CohenTaylor1978Springer} computed the stable Betti numbers of $\Conf^n\C\P^1.$ Sevryuk \cite[Theorem 2]{Sevryuk1984}, Salvatore \cite[Theorem 18]{Salvatore2004}, Randal-Williams \cite[Theorem 1.1]{Randal-Williams2013spheres}, and Ashraf, Azam, and Berceanu \cite[Lemma 6.1, Proposition 6.1]{AshrafAzamBerceanu2014} compute all the Betti numbers of $\Conf^n\C\P^1$ for $n \geq 1.$ Related computations are also done in \cite{BodigheimerCohenTaylor1989, FeichtnerZiegler2000, Napolitano2003}.

\subsection{Case of $\C\P^2$}
Recall that $\{1,x,x^2\}$ is a graded basis of $H^*(\C\P^2; \Q)$ with $|1| = 0$ and $|x| = 2.$ Set $x_i = x^i$ for $0 \leq i \leq 2.$ For $\r = (r_0,r_1) \in \Z^2$, we shall denote $\X^{\r} = \X_0^{r_0}\X_1^{r_1}.$ Since $|x_i|$ is even for $i = 0,1$ it follows that $r_i \in \{0,1\}$. Thus $B^{2p,q}(n) = \emptyset$ for $q \geq 3.$ From the description of $d$ given in Section 2, we see that 
\begin{eqnarray*}
d(\X^{(1,0)}) & = & 2X_2 + X_1^2,\\
d(\X^{(0,1)}) & = & 2X_1X_2, \quad \mbox{and}\\
d(\X^{(1,1)}) & = & 2X_2\X^{(0,1)} + X_1^2\X^{(0,1)} - 2X_1X_2\X^{(1,0)}.
\end{eqnarray*}

For $n \geq p$, $$B^{2p,0}(n) = \{X_1^p, X_1^{p-2}X_2\}$$ where $\ell(X_1^p) = p$ and $\ell(X_1^{p-2}X_2) = p-1.$ The following tables record the values of $\dim E_2^{2p,0}(n)$ for $p \geq 2$ and $0 \leq p \leq 1$, respectively.

\begin{center}
\begin{tabular}{c|c}
range of $n$ & $\dim E_2^{2p,0}(n)$  \\ \hline
$n \geq p$ & $2$\\
$n = p-1$ & $1$\\
$n \leq p-2$ & $0$
\end{tabular}
\end{center}

\begin{center}
\label{Tab:cp2dim0small}
\centering
\begin{tabular}{c|c|c}
range of $n$ & $\dim E_2^{0,0}(n)$ & $\dim E_2^{2,0}(n)$\\ \hline
$n \geq 1$ & $1$ & $1$
\end{tabular}
\end{center}

Since $E_2^{2p,-1}(n) = \{0\}$ for all $p$ and $n$, we see that $\rank d^{2p,0} = 0$.

For $n \geq p+2$, $$B^{2p,1}(n) = \{X_1^p\X^{(1,0)}, X_1^{p-2}X_1\X^{(1,0)}, X_1^{p-1}\X^{(0,1)}, X_1^{p-3}X_2\X^{(0,1)}\}$$ where $\ell(X_1^p\X^{(1,0)}) = p+2$, $\ell(X_1^{p-2}X_2\X^{(1,0)}) = \ell(X_1^{p-1}\X^{(0,1)}) = p+1,$ and $\ell(X_1^{p-3}X_2\X^{(0,1)}) = p.$ The following tables record the values of $\dim E_2^{2p,1}(n)$ for $p \geq 3$ and $0 \leq p \leq 2$, respectively.

\begin{center}
\begin{tabular}{c|c}
range of $n$ & $\dim E_2^{2p,1}(n)$\\ \hline
$n \geq p+2$ & $4$\\
$n = p+1$ & $3$\\
$n = p$ & $1$\\
$n \leq p-1$ & $0$
\end{tabular}
\end{center}

\begin{center}
\begin{tabular}{c|c|c|c}
range of $n$ & $\dim E_2^{0,1}(n)$ & $\dim E_2^{2,1}(n)$ & $\dim E_2^{4,1}(n)$\\ \hline
$n \geq 4$ & $3$ & $2$ & $1$\\
$n = 3$ & $2$ & $2$ & $1$\\
$n = 2$ & $0$ & $1$ & $1$\\
$n = 1$ & $0$ & $0$ & $0$
\end{tabular}
\end{center}

In order to compute $\rank d^{2p,1}$, we first compute $d^{2p,1}$ of each of the basis elements of $B^{2p,1}(n).$

\begin{center}
\begin{tabular}{l|l}
basis vector & image under $d^{2p,1}$\\ \hline
$X_1^p\X^{(1,0)}$ & $2X_1^pX_2 + X_1^{p+2}$\\
$X_1^{p-2}X_2\X^{(1,0)}$ & $X_1^pX_2$\\
$X_1^{p-1}\X^{(0,1)}$ & $2X_1^pX_2$\\
$X_1^{p-3}X_2\X^{(0,1)}$ & $0$
\end{tabular}
\end{center}

Thus counting the number of linearly independent vectors in the right column above, we obtain $\rank d^{2p,1}.$ The following tables record the values of $\rank d^{2p,1}$ for $p \geq 1$ and for $p = 0$, respectively.

\begin{center}
\begin{tabular}{c|c}
range of $n$ & $\rank d^{2p,1}$\\ \hline
$n \geq p+2$ & $2$\\
$n = p+1$ & $1$\\
$n \leq p$ & $0$
\end{tabular}
\end{center}

\begin{center}
\begin{tabular}{c|c}
range of $n$ & $\rank d^{0,1}$\\ \hline
$n \geq 2$ & $1$\\
$n = 1$ & $0$
\end{tabular}
\end{center}

For $n \geq p+3$, $$B^{2p,2}(n) = \left\{X_1^{p-1}\X^{(1,1)}, X_1^{p-3}X_2\X^{(1,1)}\right\}$$ where $\ell(X_1^{p-1}\X^{(1,1)}) = p+3$ and $\ell(X_1^{p-3}X_2\X^{(1,1)}) = p+2.$ The following tables record the values of $\dim E_2^{2p,2}(n)$ for $p \geq 3$ and for $0 \leq p \leq 2$, respectively.

\begin{center}
\begin{tabular}{c|c}
range of $n$ & $\dim E_2^{2p,2}(n)$\\ \hline
$n \geq p+3$ & $2$\\
$n = p+2$ & $1$\\
$n \leq p+1$ & $0$
\end{tabular}
\end{center}

\begin{center}
\begin{tabular}{c|c|c|c}
range of $n$ & $\dim E_2^{0,2}(n)$ & $\dim E_2^{2,2}(n)$ & $\dim E_2^{4,2}(n)$\\ \hline
$n \geq 5$ & $0$ & $1$ & $1$\\
$n = 4$ & $0$ & $1$ & $0$\\
$n \leq 3$ & $0$ & $0$ & $0$
\end{tabular}
\end{center}

To compute $\rank d^{2p,2}$, we first compute $d^{2p,2}$ of each of the elements of $B^{2p,2}(n).$

\begin{center}
\begin{tabular}{l|l}
basis vector & image under $d^{2p,2}$\\ \hline
$X_1^{p-1}\X^{(1,1)}$ & $2X_1^{p-1}X_2\X^{(0,1)} + X_1^{p+1}\X^{(0,1)} - 2X_1^pX_2\X^{(1,0)}$\\
$X_1^{p-3}X_2\X^{(1,1)}$ & $X_1^{p-1}X_2\X^{(0,1)}$
\end{tabular}
\end{center}

Thus counting the number of linearly independent vectors in the right column above, we obtain $\rank d^{2p,2}.$ The following tables record the values of $\rank d^{2p,2}$ for $p \geq 3$ and for $0 \leq p \leq 2$, respectively.

\begin{center}
\begin{tabular}{c|c}
range of $n$ & $\rank d^{2p,2}$\\ \hline
$n \geq p+3$ & $2$\\
$n = p+2$ & $1$\\
$n \leq p+1$ & $0$
\end{tabular}
\end{center}

\begin{center}
\begin{tabular}{c|c|c|c}
range of $n$ & $\rank d^{0,2}$ & $\rank d^{2,2}$ & $\rank d^{4,2}$\\ \hline
$n \geq 5$ & $0$ & $1$ & $1$\\
$n = 4$ & $0$ & $1$ & $0$\\
$n \leq 3$ & $0$ & $0$ & $0$
\end{tabular}
\end{center}

Now that we know $\dim E_2^{2p,q}(n)$ and $\rank d^{2p,q}$ for all $p \geq 0$ and $q = 0,1,2$, we can compute $\dim E_\infty^{2p,q}(n).$ Recall that $$\dim E_\infty^{2p,q}(n) = \dim E_2^{2p,q}(n) - \rank d^{2p,q} - \rank d^{2p-4, q+1}.$$

Following the formula above, we see that there are only finitely many values of $p$ and $q$ for which $\dim E_\infty^{2p,q}(n) \neq 0.$ We record these values in the following table.

\begin{center}
\begin{tabular}{c|c|c|c|c|c|c}
range of $n$ & $\dim E_\infty^{0,0}(n)$ & $\dim E_\infty^{2,0}(n)$ & $\dim E_\infty^{4,0}(n)$ & $\dim E_\infty^{4,1}(n)$ & $\dim E_\infty^{6,1}(n)$ & $\dim E_\infty^{8,1}(n)$\\ \hline
$n \geq 4$ & $1$ & $1$ & $1$ & $1$ & $1$ & $1$\\
$n = 3$ & $1$ & $1$ & $1$ & $1$ & $1$ & $0$\\
$n \leq 2$ & $1$ & $1$ & $1$ & $0$ & $0$ & $0$
\end{tabular}
\end{center}

Recall that $$b_i(n) = \sum_{2p+3q = i} \dim E_\infty^{2p,q}(n).$$ 

We have now proven the following proposition.

\begin{proposition}
\label{cp2betti}
The stable Betti numbers of $\Conf^n \C\P^2$ are $b_0 = b_2 = b_4 = b_7 = b_9 = b_{11} = 1$ and $b_i = 0$ for $i \neq 0,2,4,7,9,11$. Furthermore, the entire cohomology is given by 
$$
H^i(\Conf^n \C\P^2; \Q) = \begin{cases}
\Q & \mbox{if } i = 0,2,4\\
\Q & \mbox{if } n \geq 3, i = 7,9\\
\Q & \mbox{if } n \geq 4, i =11\\
0 & \mbox{else}
\end{cases}.
$$
\end{proposition}

These numbers have been previously computed. F\'{e}lix and Tanr\'{e} \cite[Theorem 2]{FelixTanre2005} computed the algebra structure of $H^*(\Conf^n\C\P^2; \Z)$ for all $n \geq 1.$ The stable Betti numbers of $\Conf^n\C\P^2$ prove Conjecture G of \cite{VakilWood2015} which was first noticed by Kupers and Miller \cite[Theorem 1.1]{KupersMiller2014} who computed the stable Betti numbers of $\Conf^n \C\P^2$ using McDuff's scanning map. Related computations are also done in \cite{Sohail2010, AshrafBerceanu2014}.

\subsection{Case of $\C\P^3$}
Recall that $H^\ast(\C\P^3; \Q) \simeq \Q[x]\slash (x^4)$ where $|1| = 0$ and $|x| = 2.$ Then if we set $x_i := x^i$ for $0 \leq i \leq 3,$ we see that $\{x_0, x_1, x_2, x_3\}$ is a graded basis of $\C\P^3.$ For $\r = (r_0, r_1, r_2) \in \Z^3,$ we shall denote $\X^\r = \X_0^{r_0}\X_1^{r_1}\X_2^{r_2}.$ Since $|x_i|$ is even for $i=0,1,2$, it follows that $r_i \in \{0,1\}$. Thus $B^{2p,q}(n) = \emptyset$ for $q \geq 4.$ From the description of $d$ given in Section 2, it follows that 
\begin{eqnarray*}
d\left(\X^{(1,0,0)}\right) & = & 2X_1X_2 + 2X_3,\\
d\left(\X^{(0,1,0)}\right) & = & 2X_1X_3 + X_2^2,\\
d\left(\X^{(0,0,1)}\right) & = & 2X_2X_3,\\
d\left(\X^{(1,1,0)}\right) & = & 2X_1X_2\X^{(0,1,0)} + 2X_3\X^{(0,1,0)} - 2X_1X_3\X^{(1,0,0)} - X_2^2\X^{(1,0,0)},\\
d\left(\X^{(1,0,1)}\right) & = & 2X_1X_2\X^{(0,0,1)} + 2X_3\X^{(0,0,1)} - 2X_2X_3\X^{(1,0,0)},\\
d\left(\X^{(0,1,1)}\right) & = & 2X_1X_3\X^{(0,0,1)} + X_2^2\X^{(0,0,1)} - 2X_2X_3\X^{(0,1,0)}, \mbox{~and}\\
d\left(\X^{(1,1,1)}\right) & = & 2X_1X_2\X^{(0,1,1)} + 2X_3\X^{(0,1,1)} - 2X_1X_3\X^{(1,0,1)} - X_2^2\X^{(1,0,1)} + 2X_2X_3\X^{(1,1,0)}.
\end{eqnarray*}

For $q = 0$, the elements of $B^{2p,0}(n)$ have no $q$-part. So we have two types of basis elements, those with an $X_3$ and those without an $X_3:$
\[
B^{2p,0}(n) = \left\{\begin{array}{l|l} X_1^{p-2j}X_2^j, & \max\{p-n,0\} \leq j \leq \floor*{\frac{p}{2}},\\ X_1^{p-3-2k}X_2^kX_3 & \max\{p-n-2,0\} \leq k \leq \floor*{\frac{p-3}{2}} \end{array}\right\}.
\]
The left-hand sides of the inequalities come from the length requirement on elements of $E_2(n)$, and the right-hand sides of the inequalities are due to the fact that the exponent of $X_1$ must be positive and integral. 

Since we have an explicit description of $B^{2p,0}(n)$, computing $\dim E_2^{2p,0}(n)$ is a just a matter of counting while taking into account the various cases presented by the inequalities, e.g. when $p - n \geq 0.$ The following tables record the values of $\dim E_2^{2p,0}(n)$ for $p \geq 4$ and for $0 \leq p \leq 3$, respectively.

\begin{center}
\begin{tabular}{c|c}
range of $n$ & $\dim E_2^{2p,0}(n)$ \\ \hline
$n \geq p$ & $p$\\
$p-1 \leq n \leq p-2$ & $n$\\
$\frac{p-2}{2} \leq n \leq p-3$ & $2n-p+2$\\
$n \leq \frac{p-3}{2}$ & $0$
\end{tabular}
\end{center}

\begin{center}
\begin{tabular}{c|c|c|c|c}
 range of $n$ & $\dim E_2^{0,0}(n)$ & $\dim E_2^{2,0}(n)$ & $\dim E_2^{4,0}(n)$ & $\dim E_2^{6,0}(n)$ \\ \hline
$n \geq 3$ & $1$ & $1$ & $2$ & $3$\\
$n = 2$ & $1$ & $1$ & $2$ & $2$\\
$n = 1$ & $1$ & $1$ & $1$ & $1$
\end{tabular}
\end{center}

Recall that for all $p$, $d^{2p,0}$ is the zero map, and so $\rank d^{2p,0} = 0$.

There are three possible $q$-parts for elements of $B^{2p,1}(n)$: $\X^{(1,0,0)},\: \X^{(0,1,0)},$ and $\X^{(0,0,1)}$. Thus when we further separate basis vectors by whether or not they have $X_3$ as a factor, we obtain six types of basis vectors:
\[
B^{2p,1}(n) = \left\{\begin{array}{l|l} X_1^{p-2j}X_2^j\X^{(1,0,0)}, & \max\{p-n+2,0\} \leq j \leq \floor*{\frac{p}{2}},\\ X_1^{p-3-2k}X_2^kX_3\X^{(1,0,0)}, & \max\{p-n,0\} \leq k \leq \floor*{\frac{p-3}{2}},\\
X_1^{p-1-2\ell}X_2^\ell\X^{(0,1,0)}, & \max\{p-n+1,0\} \leq \ell \leq \floor*{\frac{p-1}{2}},\\ X_1^{p-4-2r}X_2^rX_3\X^{(0,1,0)}, & \max\{p-n-1,0\} \leq r \leq \floor*{\frac{p-4}{2}},\\ X_1^{p-2-2s}X_2^s\X^{(0,0,1)}, & \max\{p-n,0\} \leq s \leq \floor*{\frac{p-2}{2}},\\ X_1^{p-5-2t}X_2^tX_3\X^{(0,0,1)} & \max\{p-n-2, 0\} \leq t \leq \floor*{\frac{p-5}{2}}
 \end{array}\right\}.
\]
Again, the left-hand sides of the inequalities are due to the length requirement, and the right-hand sides come from the necessity that the exponent of $X_1$ be positive and integral.

As in the case of $E_2^{2p,0}(n),$ since we have an explicit description of $B^{2p,1}(n),$ computing $\dim E_2^{2p,1}(n)$ is just a matter of counting while taking into account the various cases presented by the inequalities. The following tables record the values of $\dim E_2^{2p,1}(n)$ for $p \geq 6$ and for $0 \leq p \leq 5$, respectively.

\begin{center}
\begin{tabular}{c|c}
range of $n$ & $\dim E_2^{2p,1}$ \\ \hline
$n \geq p+2$ & $3p-3$\\
$n = p+1$ & $3p-4$\\
$n=p$ & $3p-6$\\
$n=p-1$ & $3p-10$\\
$\frac{p+2}{2} \leq n \leq p-2$ & $6n-3p-3$\\
$n = \frac{p+1}{2}$ & $1$\\
$n\leq \frac{p}{2}$ & $0$
\end{tabular}
\end{center}

\begin{center}
\begin{tabular}{c|c|c|c|c|c|c}
range of $n$ & $\dim E_2^{0,1}(n)$ & $\dim E_2^{2,1}(n)$ & $\dim E_2^{4,1}(n)$ & $\dim E_2^{6,1}(n)$ & $\dim E_2^{8,1}(n)$ & $\dim E_2^{10,1}(n)$\\ \hline
$n \geq 7$ & $1$ & $2$ & $4$ & $6$ & $9$ & $12$\\
$n = 6$ & $1$ & $2$ & $4$ & $6$ & $9$ & $11$\\
$n = 5$ & $1$ & $2$ & $4$ & $6$ & $8$ & $9$\\
$n = 4$ & $1$ & $2$ & $4$ & $5$ & $6$ & $5$\\
$n = 3$ & $1$ & $2$ & $3$ & $3$ & $2$ & $1$\\
$n = 2$ & $1$ & $1$ & $1$ & $0$ & $0$ & $0$\\
$n = 1$ & $0$ & $0$ & $0$ & $0$ & $0$ & $0$
\end{tabular}
\end{center}

To compute the rank of $d^{2p,1}$, we first compute $d^{2p,1}$ of each of the basis vectors of $B^{2p,1}(n).$

\begin{center}
\begin{tabular}{l|l}
basis vector & image under $d^{2p,1}$\\ \hline
$X_1^{p-2j}X_2^j\X^{(1,0,0)}$ & $2X_1^{p+1-2j}X_2^{j+1} + 2X_1^{p-2j}X_2^jX_3$\\
$X_1^{p-3-2k}X_2^kX_3\X^{(1,0,0)}$ & $2X_1^{p-2-2k}X_2^{k+1}X_3$\\
$X_1^{p-1-2\ell}X_2^{\ell}\X^{(0,1,0)}$ & $2X_1^{p-2\ell}X_2^{\ell}X_3 + X_1^{p-1-2\ell}X_2^{\ell+2}$\\
$X_1^{p-4-2r}X_2^rX_3\X^{(0,1,0)}$ & $X_1^{p-4-2r}X_2^{r+2}X_3$\\
$X_1^{p-2-2s}X_2^s\X^{(0,0,1)}$ & $2X_1^{p-2-2s}X_2^{s+1}X_3$\\
$X_1^{p-5-2t}X_2^tX_3\X^{(0,0,1)}$ & $0$
\end{tabular}
\end{center}

We then count the number of linearly independent vectors in the image of $d^{2p,1}$ to obtain $\rank d^{2p,1}.$ The following tables record the values of $\rank d^{2p,1}$ for $p \geq 2$ and for $p = 0,1$, respectively.

\begin{center}
\begin{tabular}{c|c}
range of $n$ & $\rank d^{2p,1}$\\ \hline
$n \geq p+2$ & $p+2$\\
$\frac{p+2}{2} \leq n \leq p+1$ & $2n-p-1$\\
$n \leq \frac{p+1}{2}$ & $0$
\end{tabular}
\end{center}

\begin{center}
\begin{tabular}{c|c|c}
range of $n$ & $\rank d^{0,1}$ & $\rank d^{2,1}$\\ \hline
$n \geq 3$ & $1$ & $2$\\
$n = 2$ & $1$ & $1$\\
$n = 1$ & $0$ & $0$
\end{tabular}
\end{center}

For elements of $B^{2p,2}(n)$, there are three possible $q$-parts: $\X^{(1,1,0)},$ $\X^{(1,0,1)}$, and $\X^{(0,1,1)}$. Thus when we further distinguish basis vectors by whether or not they have $X_3$ as a factor, we have six types of basis vectors:
\[
B^{2p,2}(n) = \left\{\begin{array}{l|l} X_1^{p-1-2j}X_2^j\X^{(1,1,0)}, & \max\{p-n+3,0\} \leq j \leq \floor*{\frac{p-1}{2}},\\ X_1^{p-4-2k}X_2^kX_3\X^{(1,1,0)}, & \max\{p-n+1,0\} \leq k \leq \floor*{\frac{p-4}{2}},\\
X_1^{p-2-2\ell}X_2^\ell\X^{(1,0,1)}, & \max\{p-n+2,0\} \leq \ell \leq \floor*{\frac{p-2}{2}},\\ X_1^{p-5-2r}X_2^rX_3\X^{(1,0,1)}, & \max\{p-n,0\} \leq r \leq \floor*{\frac{p-5}{2}},\\ X_1^{p-3-2s}X_2^s\X^{(0,1,1)}, & \max\{p-n+1,0\} \leq s \leq \floor*{\frac{p-3}{2}},\\ X_1^{p-6-2t}X_2^tX_3\X^{(0,1,1)} & \max\{p-n-1,0\} \leq t \leq \floor*{\frac{p-6}{2}} \end{array}\right\}.
\]

As we have done previously, we count the elements of $B^{2p,2}(n)$ to determine $\dim E_2^{2p,2}(n).$ The following tables record the values of $\dim E_2^{2p,2}(n)$ for $p \geq 7$ and for $0 \leq p \leq 6$, respectively.

\begin{center}
\begin{tabular}{c|c}
range of $n$ & $\dim E_2^{2p,2}(n)$\\ \hline
$n \geq p+3$ & $3p-6$\\
$n = p+2$ & $3p-7$\\
$n = p+1$ & $3p-9$\\
$n = p$ & $3p-13$\\
$\frac{p+5}{2} \leq n \leq p-1$ & $6n-3p-12$\\
$n = \frac{p+4}{2}$ & $1$\\
$n \leq \frac{p+3}{2}$ & $0$
\end{tabular}
\end{center}

\begin{center}
\begin{tabular}{c|c|c|c|c|c|c|c}
range of $n$ & $\dim E_2^{0,2}$ & $\dim E_2^{2,2}$ & $\dim E_2^{4,2}$ & $\dim E_2^{6,2}$ & $\dim E_2^{8,2}$ & $\dim E_2^{10,2}$ & $\dim E_2^{12,2}$\\ \hline
$n \geq 9$ & $0$ & $1$ & $2$ & $4$ & $6$ & $9$ & $12$\\
$n = 8$ & $0$ & $1$ & $2$ & $4$ & $6$ & $9$ & $11$\\
$n = 7$ & $0$ & $1$ & $2$ & $4$ & $6$ & $8$ & $9$\\
$n = 6$ & $0$ & $1$ & $2$ & $4$ & $5$ & $6$ & $5$\\
$n = 5$ & $0$ & $1$ & $2$ & $3$ & $3$ & $2$ & $1$\\
$n = 4$ & $0$ & $1$ & $1$ & $0$ & $0$ & $0$ & $0$\\
$n = 3$ & $0$ & $0$ & $0$ & $0$ & $0$ & $0$ & $0$
\end{tabular}
\end{center}

To compute the rank of $d^{2p,2},$ we first compute $d^{2p,2}$ of each  of the elements of $B^{2p,2}(n).$

\begin{center}
\begin{tabular}{l|l}
basis vector & image under $d^{2p,2}$\\ \hline
$X_1^{p-1-2j}X_2^j\X^{(1,1,0)}$ & $~ 2X_1^{p-2j}X_2^{j+1}\X^{(0,1,0)} + 2X_1^{p-1-2j}X_2^jX_3\X^{(0,1,0)}$\\
~ & $- 2X_1^{p-2j}X_2^jX_3\X^{(1,0,0)} - X_1^{p-1-2j}X_2^{j+2}\X^{(1,0,0)}$\\
$X_1^{p-4-2k}X_2^kX_3\X^{(1,1,0)}$ & $~ 2X_1^{p-3-2k}X_2^{k+1}X_3\X^{(0,1,0)} - X_1^{p-4-2k}X_2^{k+2}X_3\X^{(1,0,0)}$\\
$X_1^{p-2-2\ell}X_2^\ell\X^{(1,0,1)}$ & $~ 2X_1^{p-1-2\ell}X_2^{\ell+1}\X^{(0,0,1)} + 2X_1^{p-2-2\ell}X_2^\ell X_3\X^{(0,0,1)}$\\
~ & $- 2X_1^{p-2-2\ell}X_2^{\ell+1}X_3\X^{(1,0,0)}$\\
$X_1^{p-5-2r}X_2^rX_3\X^{(1,0,1)}$ & $~ 2X_1^{p-4-2r}X_2^{r+1}X_3\X^{(0,0,1)}$\\
$X_1^{p-3-2s}X_2^s\X^{(0,1,1)}$ & $~ 2X_1^{p-2-2s}X_2^sX_3\X^{(0,0,1)} + X_1^{p-3-2s}X_2^{s+2}\X^{(0,0,1)}$\\
~ & $- 2X_1^{p-3-2s}X_2^{s+1}X_3\X^{(0,1,0)}$\\
$X_1^{p-6-2t}X_2^tX_3\X^{(0,1,1)}$ & $~ X_1^{p-6-2t}X_2^{t+2}X_3\X^{(0,0,1)}$
\end{tabular}
\end{center}

We then count the number of linearly independent vectors in the image of $d^{2p,2}$ to obtain $\rank d^{2p,2}.$ The following tables record the values of $\rank d^{2p,2}$ for $p \geq 5$ and for $0 \leq p \leq 4$, respectively.

\begin{center}
\begin{tabular}{c|c}
range of $n$ & $\rank d^{2p,2}$\\ \hline
$n \geq p+3$ & $2p-1$\\
$n = p+2$ & $2p-2$\\
$n = p+1$ & $2p-4$\\
$\frac{p+5}{2} \leq n \leq p$ & $4n-2p-8$\\
$n = \frac{p+4}{2}$ & $1$\\
$n \leq \frac{p+3}{2}$ & $0$
\end{tabular}
\end{center}

\begin{center}
\begin{tabular}{c|c|c|c|c|c}
range of $n$ & $\rank d^{0,2}$ & $\rank d^{2,2}$ & $\rank d^{4,2}$ & $\rank d^{6,2}$ & $\rank d^{8,2}$\\ \hline
$n \geq 10$ & $0$ & $1$ & $2$ & $4$ & $6$\\
$n = 9$ & $0$ & $1$ & $2$ & $4$ & $6$\\
$n = 8$ & $0$ & $1$ & $2$ & $4$ & $6$\\
$n = 7$ & $0$ & $1$ & $2$ & $4$ & $6$\\
$n = 6$ & $0$ & $1$ & $2$ & $4$ & $5$\\
$n = 5$ & $0$ & $1$ & $2$ & $3$ & $3$\\
$n = 4$ & $0$ & $1$ & $1$ & $0$ & $0$\\
$n\leq 3$ & $0$ & $0$ & $0$ & $0$ & $0$
\end{tabular}
\end{center}

There is one possible $q$-part for elements of $B^{2p,3}(n)$, namely $\X^{(1,1,1)}.$  Thus the basis vectors of $B^{2p,3}(n)$ are differentiated by whether or not they have an $X_3$:
\[B^{2p,3}(n) = \left\{\begin{array}{l|l} X_1^{p-3-2j}X_2^j\X^{(1,1,1)}, & \max\{p-n+3,0\} \leq j \leq \floor*{\frac{p-3}{2}},\\ X_1^{p-6-2k}X_2^kX_3\X^{(1,1,1)} & \max\{p-n+1, 0\} \leq k \leq \floor*{\frac{p-6}{2}} \end{array}\right\}. \]

To determine $\dim E_2^{2p,3}(n)$, we count the elements of $B^{2p,3}(n).$ The following tables record the values of $\dim E_2^{2p,3}(n)$ for $p \geq 6$ and for $0 \leq p \leq 5$, respectively.

\begin{center}
\begin{tabular}{c|c}
range of $n$ & $\dim E_2^{2p,3}(n)$\\ \hline
$n \geq p+3$ & $p-3$\\
$n = p+2$ & $p-4$\\
$\frac{p+8}{2} \leq n \leq p+1$ & $2n-p-7$\\
$n \leq \frac{p+7}{2}$ & $0$
\end{tabular}
\end{center}

\begin{center}
\begin{tabular}{c|c|c|c|c|c|c}
range of $n$ & $\dim E_2^{0,3}(n)$ & $\dim E_2^{2,3}(n)$ & $\dim E_2^{4,3}(n)$ & $\dim E_2^{6,3}(n)$ & $\dim E_2^{8,3}(n)$ & $\dim E_2^{10,3}(n)$\\ \hline
$n \geq 8$ & $0$ & $0$ & $0$ & $1$ & $1$ & $2$\\
$n = 7$ & $0$ & $0$ & $0$ & $1$ & $1$ & $1$\\
$n = 6$ & $0$ & $0$ & $0$ & $1$ & $0$ & $0$\\
$n \leq 5$ & $0$ & $0$ & $0$ & $0$ & $0$ & $0$
\end{tabular}
\end{center}

Next we compute $d^{2p,3}$ of each of the basis elements of $B^{2p,3}(n).$

\begin{center}
\begin{tabular}{l|l}
basis vector & image under $d^{2p,3}$\\ \hline
$X_1^{p-3-2j}X_2^j\X^{(1,1,1)}$ & $~ 2X_1^{p-2-2j}X_2^{j+1}\X^{(0,1,1)} + 2X_1^{p-3-2j}X_2^jX_3\X^{(0,1,1)}$\\
~ & $- 2X_1^{p-2-2j}X_2^jX_3\X^{(1,0,1)} - X_1^{p-3-2j}X_2^{j+2}\X^{(1,0,1)}$\\
~ & $+ 2X_1^{p-3-2j}X_2^{j+1}X_3\X^{(1,1,0)}$\\
$X_1^{p-6-2k}X_2^kX_3\X^{(1,1,1)}$ & $~ 2X_1^{p-5-2k}X_2^{k+1}X_3\X^{(0,1,1)} - X_1^{p-6-2k}X_2^{k+2}X_3\X^{(1,0,1)}$
\end{tabular}
\end{center}

Note that $d^{2p,3}\left(B^{2p,3}(n)\right)$ is a linearly independent set, so $\rank d^{2p,3} = \dim E_2^{2p,3}(n)$ for all $p$. The following tables record the values of $\rank d^{2p,3}$ for $p \geq 6$ and for $0 \leq p \leq 5$, respectively.

\begin{center}
\begin{tabular}{c|c}
range of $n$ & $\rank d^{2p,3}$\\ \hline
$n \geq p+3$ & $p-3$\\
$n = p+2$ & $p-4$\\
$\frac{p+8}{2} \leq n \leq p+1$ & $2n-p-7$\\
$n \leq \frac{p+7}{2}$ & $0$
\end{tabular}
\end{center}

\begin{center}
\begin{tabular}{c|c|c|c|c|c|c}
range of $n$ & $\rank d^{0,3}$ & $\rank d^{2,3}$ & $\rank d^{4,3}$ & $\rank d^{6,3}$ & $\rank d^{8,3}$ & $\rank d^{10,3}$\\ \hline
$n \geq 8$ & $0$ & $0$ & $0$ & $1$ & $1$ & $2$\\
$n = 7$ & $0$ & $0$ & $0$ & $1$ & $1$ & $1$\\
$n = 6$ & $0$ & $0$ & $0$ & $1$ & $0$ & $0$\\
$n \leq 5$ & $0$ & $0$ & $0$ & $0$ & $0$ & $0$
\end{tabular}
\end{center}

Now that we know $\dim E_2^{2p,q}(n)$ and $\rank d^{2p,q}$ for all $p \geq 0$ and $q = 0,1,2,3$, we can compute $\dim E_\infty^{2p,q}(n)$. Recall that $$\dim E_\infty^{2p,q}(n) = \dim E_2^{2p,q}(n) - \rank d^{2p,q} - \rank d^{2p-6,q+1}.$$ The following tables record the values of $\dim E_\infty^{2p,0}(n)$ for $p \geq 5$ and for $0 \leq p \leq 4$, respectively.

\begin{center}
\begin{tabular}{c|c}
range of $n$ & $\dim E_\infty^{2p,0}(n)$\\ \hline
$n \geq p$ & $1$\\
$n \leq p-1$ & $0$
\end{tabular}
\end{center}

\begin{center}
\begin{tabular}{c|c|c|c|c|c}
range of $n$ & $\dim E_\infty^{0,0}(n)$ & $\dim E_\infty^{2,0}(n)$ & $
\dim E_\infty^{4,0}(n)$ & $\dim E_\infty^{6,0}(n)$ & $\dim E_\infty^{8,0}(n)$\\ \hline
$n \geq 4$ & $1$ & $1$ & $2$ & $2$ & $2$\\
$n = 3$ & $1$ & $1$ & $2$ & $2$ & $1$\\
$n = 2$ & $1$ & $1$ & $2$ & $1$ & $1$\\
$n = 1$ & $1$ & $1$ & $1$ & $1$ & $0$
\end{tabular}
\end{center}

The following tables record the values of $\dim E_\infty^{2p,1}(n)$ for $p \geq 8$ and for $0 \leq p \leq 7$, respectively.

\begin{center}
\begin{tabular}{c|c}
range of $n$ & $\dim E_\infty^{2p,1}(n)$, \\ \hline
 $n \geq p$ & $2$\\
 $n = p-1$ & $1$\\
 $n \leq p-2$ & $0$
\end{tabular}
\end{center}

\begin{center}
\begin{tabular}{c|c|c|c|c|c|c|c|c}
range of $n$ & $\dim E_\infty^{0,1}$ & $\dim E_\infty^{2,1}$ & $\dim E_\infty^{4,1}$ & $\dim E_\infty^{6,1}$ & $\dim E_\infty^{8,1}$ & $\dim E_\infty^{10,1}$ & $\dim E_\infty^{12,1}$ & $\dim E_\infty^{14,1}$\\ \hline
$n \geq 7$ & $0$ & $0$ & $0$ & $1$ & $2$ & $3$ & $3$ & $3$\\
$n = 6$ & $0$ & $0$ & $0$ & $1$ & $2$ & $3$ & $3$ & $2$\\
$n = 5$ & $0$ & $0$ & $0$ & $1$ & $2$ & $3$ & $2$ & $1$\\
$n = 4$ & $0$ & $0$ & $0$ & $1$ & $2$ & $2$ & $1$ & $1$\\
$n = 3$ & $0$ & $0$ & $0$ & $1$ & $1$ & $1$ & $0$ & $0$\\
$n \leq 2$ & $0$ & $0$ & $0$ & $0$ & $0$ & $0$ & $0$ & $0$
\end{tabular}
\end{center}

The following table records the values of $\dim E_\infty^{2p,2}(n)$ for $p \geq 7.$

\begin{center}
\begin{tabular}{c|c}
range of $n$ & $\dim E_\infty^{2p,2}(n)$\\ \hline
$n \geq p-1$ & $1$\\
$n \leq p-2$ & $0$
\end{tabular}
\end{center}

For $0 \leq p \leq 6$ and $n \geq 1$, $\dim E_\infty^{2p,2}(n) = 0$. Since $\dim E_2^{2p,3}(n) = \rank d^{2p,3}$ and $d^{2p,4}$ is the zero map, it follows that $\dim E_\infty^{2n,3}(n) = 0$ for all $n$ and $p$.

Recall that for $i$ even, $$b_i(n) = \dim E_\infty^{i,0}(n) + \dim E_\infty^{i-10,2}(n).$$
and for $i$ odd, $$b_i(n) = \dim E_\infty^{i-5,1}(n) + \dim E_\infty^{i-15,3}(n).$$

We have now proven the following proposition, as well as Proposition \ref{cp3stableinstability}.

\begin{proposition}
\label{cp3betti}
The stable Betti numbers of $\Conf^n \C\P^3$ are $b_i = 2$ for $i \geq 23$. Furthermore, the entire cohomology is given by
$$
H^i(\Conf^n \C\P^3; \Q) = \begin{cases}
\Q^3 & \mbox{if } n \geq \frac{i-5}{2}, i =15,17,19\\
\Q^2 & \mbox{if } n = \frac{i-7}{2}, i = 15,17,19\\
\Q^2 & \mbox{if } n \geq \frac{i}{2}, i\mbox{ even }, i \geq 24 \mbox{ or } i = 4,8\\
\Q^2 & \mbox{if } n\geq \frac{i-5}{2}, i \mbox{ odd }, i \geq 21 \mbox{ or } i = 13,21\\
\Q & \mbox{if } \frac{i-12}{2} \leq n \leq \frac{i-2}{2}, i \mbox{ even } i \geq 24\\
\Q & \mbox{if } n \geq \frac{i}{2}, i \mbox{ even }, 10 \leq i \leq 22,\\
\Q & \mbox{if } i = 0,2\\
\Q & \mbox{if } n = 1, i=4\\
\Q & \mbox{if } n =1,2, i = 6\\
\Q & \mbox{if } n = 2,3, i =8\\
\Q & \mbox{if } n = \frac{i-7}{2}, i \geq 21 \mbox{ or } i =13\\
\Q & \mbox{if } n = \frac{i-9}{2}, i = 15,17,19\\
\Q & \mbox{if } n = \frac{i-11}{2}, i = 15,19\\
\Q & \mbox{if } n \geq 3, i = 11\\
0 & \mbox{else}
\end{cases}.
$$
\end{proposition}

The stable Betti numbers of $\Conf^n\C\P^3$ prove Conjecture H of \cite{VakilWood2015}. F\'{e}lix and Thomas \cite[Section 3.2]{FelixThomas2000} computed $H_i(\Conf^3\C\P^3; \Q)$ for $0 \leq i \leq 15,$ and later Ashraf and Berceanu \cite[Theorem 1.4]{AshrafBerceanu2014} computed the cohomology algebra $H^*(\Conf^3\C\P^3;\Q).$ But Kupers and Miller \cite[Theorem 1.1]{KupersMiller2014} were the first to verify Conjecture H by computing all the stable Betti numbers of $\Conf^n\C\P^3$ using McDuff's scanning map.

\subsection{Genus 1 Riemann surface}
Recall that $\{1,a,b,t\}$ is a graded basis of $H^*(\Sigma_1; \Q)$ where $|a| = |b| = 1,$ $|t| = 2$, and $ab = -ba = t.$ An arbitrary basis vector of $E_2^{p,q}(n)$ is of the form $A^iB^jT^k\1^\ell\A^r\B^s$ where $i+j+r+s+k = p$, $\ell +r+s = q$, and $i+j+k+2\ell+2r+2s \leq n.$ But since $|1|=0$ is even and $|a|=|b|=1$ is odd, it follows that $i,j,\ell \in \{0,1\}.$ Furthermore, since $t$ is the orientation class of $H^*(X; \Q)$, we have that $k \in \{0,1\}$. Thus for a fixed $q$, it follows that $E_2^{p,q}(n) = \{0\}$ for $p \leq q-2$ and $p \geq q+5.$ 

From the description of $d$ given in Section 2, we have that
\begin{eqnarray*}
d\left(\1\A^j\B^{q-1-j}\right) & = & 2T\A^j\B^{q-1-j} - 2AB\A^j\B^{q-1-j} - 2jAT\1\A^{j-1}\B^{q-1-j}\\
&~& - 2(q-1-j)T\1\A^j\B^{q-2-j} \quad \mbox{and}\\
d\left(\A^k\B^{q-k}\right) & = & -2kAT\A^{k-1} - 2(q-k)BT\A^k\B^{q-1-k}.
\end{eqnarray*}

For $q \geq 1$ and $n\geq 2q$,
$$B^{q-1,q}(n) = \{\1\A^j\B^{q-j-1} : 0 \leq j \leq q-1\}.$$
Since $\ell(\1\A^j\B^{q-1-j}) = 2q$, we see that $B^{q-1,q}(n) = \emptyset$ for $n \leq 2q-1$.

The following table records the values of $\dim E_2^{q-1,q}(n)$ for $q \geq 1.$

\begin{center}
\begin{tabular}{c|c}
range of $n$ & $\dim E_2^{q-1,q}(n)$\\ \hline
$n \geq 2q$ & $q$\\
$n \leq 2q-1$ & $0$
\end{tabular}
\end{center}

To compute the rank of $d^{q-1,q}$, we compute the image of each basis vector of $B^{q-1,q}(n)$.

\begin{center}
\begin{tabular}{l|l}
basis vector & image under $d^{q-1,q}$\\ \hline
$\1\A^j\B^{q-1-j}$ & $2T\A^j\B^{q-1-j} - 2AB\A^j\B^{q-1-j} - 2jAT\1\A^{j-1}\B^{q-1-j}$\\
~ & $- 2(q-1-j)BT\1\A^j\B^{q-2-j}$
\end{tabular}
\end{center}

We count the number of linearly independent vectors in the rightmost column above to obtain $\rank d^{q-1,q}.$ The following table records the values of $d^{q-1,q}$ for $q \geq 1.$

\begin{center}
\begin{tabular}{c|c}
range of $n$ & $\rank d^{q-1,q}$  \\ \hline
$n \geq 2q$ & $q$\\
$n \leq 2q-1$ & $0$
\end{tabular}
\end{center}

For $q \geq 1$ and $n \geq 2q+1$,
$$B^{q,q}(n) = \{ A\1\A^j\B^{q-1-j}, B\1\A^j\B^{q-1-j}, \A^k\B^{q-k} : 0 \leq j \leq q-1, 0 \leq k \leq q\},$$
where $\ell(A\1\A^j\B^{q-1-j}) = \ell(B\1\A^j\B^{q-1-j}) = 2q+1$ and $\ell(\A^k\B^{q-k}) = 2q$. The following table records the values of $\dim E_2^{q,q}(n)$ for $q \geq 1.$

\begin{center}
\begin{tabular}{c|c}
range of $n$ & $\dim E_2^{q,q}(n)$\\ \hline
$n \geq 2q+1$ & $3q+1$\\
$n = 2q$ & $q+1$\\
$n \leq 2q-1$ & $0$
\end{tabular}
\end{center}

To compute the rank of $d^{q,q}$, we first compute the image of each element of $B^{q,q}(n)$.

\begin{center}
\begin{tabular}{l|l}
basis vector & image under $d^{q,q}$\\ \hline
$A\1\A^j\B^{q-1-j}$ & $2(q-1-j)ABT\1\A^j\B^{q-2-j} - 2AT\A^j\B^{q-1-j}$\\
$B\1\A^j\B^{q-1-j}$ & $-2jABT\A^{j-1}\B^{q-1-j} - 2BT\A^j\B^{q-1-j}$\\
$\A^k\B^k$ & $-2kAT\A^{k-1}\B^{q-k} - 2(q-k)BT\A^k\B^{q-1-k}$
\end{tabular}
\end{center}

Thus counting the number or linearly independent vectors in the image, we obtain $\rank d^{q,q}.$ The following table records the values of $\rank d^{q,q}$ for $q \geq 1.$

\begin{center}
\begin{tabular}{c|c}
range of $n$ & $\rank d^{q,q}$  \\ \hline
$n \geq 2q+1$ & $2q$\\
$n = 2q$ & $q+1$\\
$n \leq 2q-1$ & $0$
\end{tabular}
\end{center}

For $q \geq 1$ and $n \geq 2q+2$, we have
$$B^{q+1,q}(n) = \{AB\1\A^j\B^{q-1-j}, T\1\A^j\B^{q-1-j}, A\A^k\B^{q-k}, B\A^k\B^{q-k} : 0 \leq j \leq q-1, 0\leq k \leq q\},$$ where $\ell(AB\1\A^j\B^{q-1-j}) = 2q+2$ and $\ell(T\1\A^j\B^{q-1-j}) = \ell(A\A^k\B^{q-k}) = \ell(B\A^k\B^{q-k}) = 2q+1$. The following table records the values of $\dim E_2^{q+1,q}(n)$ for $ q \geq 1.$

\begin{center}
\begin{tabular}{c|c}
range of $n$ & $\dim E_2^{q+1,q}(n)$\\ \hline
$n \geq 2q+2$ & $4q+2$\\
$n = 2q+1$ & $3q+2$\\
$n \leq 2q$ & $0$
\end{tabular}
\end{center}

We compute the image under $d^{q+1,q}$ of each element of $B^{q+1,q}$.

\begin{center}
\begin{tabular}{l|l}
basis vector & image under $d^{q+1,q}$\\ \hline
$AB\1\A^j\B^{q-1-j}$ & $2ABT\A^j\B^{q-1-j}$\\
$T\1\A^j\B^{q-1-j}$ & $-2ABT\A^j\B^{q-1-j}$\\
$A\A^k\B^{q-k}$ & $-2kABT\A^{k-1}\B^{q-k}$\\
$B\A^k\B^{q-k}$ & $2(q-k)ABT\A^k\B^{q-k}$
\end{tabular}
\end{center}

Thus counting the number of linearly independent vectors in the image, we obtain $\rank d^{q+1,q}.$ The following table records the values of $d^{q+1,q}$ for $q \geq 1.$

\begin{center}
\begin{tabular}{c|c}
range of $n$ & $\rank d^{q+1,q}$\\ \hline
$n \geq 2q+1$ & $q$\\
$n \leq 2q$ & $0$
\end{tabular}
\end{center}

For $q \geq 1$ and $n \geq 2q+2$, we have
$$B^{q+2,q}(n) = \{AT\1\A^j\B^{q-1-j}, BT\1\A^j\B^{q-1-j}, AB\A^k\B^{q-k}, T\A^k\B^{q-k}: 0\leq j \leq q-1, 0 \leq k \leq q\},$$
where $\ell(AT\1\A^j\B^{q-1-j}) = \ell(BT\1\A^j\B^{q-1-j}) = \ell(AB\A^k\B^{q-k}) = 2q+2$ and $\ell(T\A^k\B^{q-k}) = 2q+1$. The following table records the values of $\dim E_2^{q+2,q}(n)$ for $q \geq 1.$

\begin{center}
\begin{tabular}{c|c}
range of $n$ & $\dim E_2^{q+2,q}(n)$\\ \hline
$n \geq 2q+2$ & $4q+2$\\
$n = 2q+1$ & $q+1$\\
$n \leq 2q$ & $0$
\end{tabular}
\end{center}

Since $d^{q+2,q}: E_2^{q+2,q}(n) \rightarrow E_2^{q+4,q-1}(n) = \{0\}$, it follows that $\rank d^{q+2,q} = 0$ for all $q \geq 0$ and $n \geq 1.$

For $q \geq 1$ and $n \geq 2q+3$, we have
$$B^{q+3,q}(n) = \{ABT\1\A^j\B^{q-1-j}, AT\A^k\B^{q-k}, BT\A^k\B^{q-k} : 0 \leq j \leq q-1, 0 \leq q \leq k\},$$ where $\ell(ABT\1\A^j\B^{q-1-j}) = 2q+3$ and $\ell(AT\A^k\B^{q-k}) = \ell(BT\A^k\B^{q-k}) = 2q+2$. The following table records the values of $\dim E_2^{q+3,q}(n)$ for $q \geq 1.$

\begin{center}
\begin{tabular}{c|c}
range of $n$ & $\dim E_2^{q+3,q}(n)$\\ \hline
$n \geq 2q+3$ & $3q+2$\\
$n = 2q+2$ & $2q+2$\\
$n \leq 2q+1$ & $0$
\end{tabular}
\end{center}

Since $d^{q+3,q}: E_2^{q+3,q}(n) \rightarrow E_2^{q+5,q-1}(n) = \{0\},$ we see that $\rank d^{q+3,q} = 0$ for all $q \geq 0$ and $n \geq 1.$

For $q \geq 1$ and $n \geq 2q+3$, we have
$$B^{q+4,q} = \{ABT\A^k\B^{q-k} : 0\leq k \leq q\},$$ where $\ell(ABT\A^k\B^{q-k}) = 2q+3$. The following table records the values of $\dim E_2^{q+4,q}(n)$ for $q \geq 1$.

\begin{center}
\begin{tabular}{c|c}
 range of $n$ & $\dim E_2^{q+4,q}(n)$ \\ \hline
$n \geq 2q+3$ & $q+1$\\
$n \leq 2q+2$ & $0$
\end{tabular}
\end{center}

Since $d^{q+4,q}: E_2^{q+4,q}(n) \rightarrow E_2^{q+6, q-1} = \{0\}$, it follows that $\rank d^{q+4,q} = 0$ for all $q \geq 0$ and $n \geq 1.$

We have yet to consider the $q=0$ case. For $q = 0$ and $n \geq 3$, we have
\begin{eqnarray*}
B^{0,0}(n) & = & \{1\},\\
B^{1,0}(n) & = & \{A,B\},\\
B^{2,0}(n) & = & \{AB, T\},\\
B^{3,0}(n) & = & \{AT, BT\},\quad \mbox{and}\\
B^{4,0}(n) & = & \{ABT\},
\end{eqnarray*}
where $\ell(1) = \ell(A) = \ell(B) = \ell(T) = 1$, $\ell(AB) = \ell(AT) = \ell(BT) = 2$, and $\ell(ABT) = 3.$ From this we obtain the following table which records the values of $\dim E_2^{p,0}(n)$ for $0 \leq p \leq 4.$

\begin{center}
\begin{tabular}{c|c|c|c|c|c}
range of $n$ & $\dim E_2^{0,0}(n)$ & $\dim E_2^{1,0}(n)$ & $\dim E_2^{2,0}(n)$ & $\dim E_2^{3,0}(n)$ & $\dim E_2^{4,0}(n)$\\ \hline
$n \geq 3$ & $1$ & $2$ & $2$ & $2$ & $1$\\
$n = 2$ & $1$ & $2$ & $2$ & $2$ & $0$\\
$n = 1$ & $1$ & $2$ & $1$ & $0$ & $0$
\end{tabular}
\end{center}

Since $E_2^{p,-1}(n) = \{0\}$, we have that $\rank d^{p,0} = 0$ for all $p \geq 0$ and $n \geq 1.$

Now that we know $\dim E_2^{p,q}(n)$ and $\rank d^{p,q}$, we can compute $\dim E_\infty(n)$ for all $q \geq 0$ and $q-1 \leq p \leq q+4.$ Recall that $$\dim E_\infty^{p,q}(n) = \dim E_2^{p,q}(n) - \rank d^{p,q} - \rank d^{p-2, q+1}.$$ Since $\dim E_2^{q-1,q}(n) = \rank d^{q-1,q}$, we have that $\dim E_\infty^{q-1,q}(n) = 0$ for all $q \geq 0$ and $n \geq 1.$ Additionally, since $\dim E_2^{q+4,q}(n) = \rank d^{q+2, q+1}$, it follows that $\dim E_\infty^{q+4,q}(n) = 0$ for all $q \geq 0$ and $n \geq 1.$ The following tables record the values of $\dim E_\infty^{p,q}(n)$ for $q \geq 1$ and $q \leq p \leq q+3$ and for $q\geq 1$ and $0 \leq p \leq 4$, respectively.

\begin{center}
\begin{tabular}{c|c|c|c|c}
range of $n$ & $\dim E_\infty^{q,q}(n)$ & $\dim E_\infty^{q+1,q}(n)$ & $\dim E_\infty^{q+2,q}(n)$ & $\dim E_\infty^{q+3,q}(n)$ \\ \hline
$n \geq 2q+2$ & $q+1$ & $3q+2$ & $3q+1$ & $q$\\
$n = 2q+1$ & $q+1$ & $2q+2$ & $q+1$ & $0$\\
$n \leq 2q$ & $0$ & $0$ & $0$ & $0$
\end{tabular}
\end{center}

\begin{center}
\begin{tabular}{c|c|c|c|c|c}
range of $n$ & $\dim E_\infty^{0,0}(n)$ & $\dim E_\infty^{1,0}(n)$ & $\dim E_\infty^{2,0}(n)$ & $\dim E_\infty^{3,0}(n)$ & $\dim E_\infty^{4,0}(n)$ \\ \hline
$n \geq 1$ & $1$ & $2$ & $1$ & $0$ & $0$
\end{tabular}
\end{center}

Recall that $$b_i(n) = \sum_{p+q = i} \dim E_\infty^{p,q}(n).$$ Thus for $i$ even, $$b_i(n) = \dim E_\infty^{\frac{i}{2},\frac{i}{2}}(n) + \dim E_\infty^{\frac{i+2}{2}, \frac{i-2}{2}}(n),$$ and for $i$ odd, $$b_i(n) = \dim E_\infty^{\frac{i+1}{2}, \frac{i-1}{2}}(n) + \dim E_\infty^{\frac{i+3}{2}, \frac{i-3}{2}}(n).$$ 

We have now proven the following propostion, as well as Proposition \ref{stableg1betti}.

\begin{proposition}[Cohomology of the elliptic braid group]
\label{P:g1SV}
The stable Betti numbers of $\Conf^n \Sigma_1$ are $$b_0 =1, b_1 = 2, b_2 = 3, b_3 = 5, b_ 4 = 7, \hdots, b_i = 2i-1,\hdots.$$ Furthermore, the entire cohomology is given by 
$$
H^{i}(\Conf^n \Sigma_1; \Q)= \begin{cases}
\Q^{2i-1} & \mbox{if } n \geq i+1, i \geq 2\\
\Q^{\frac{3i-4}{2}} & \mbox{if } n = i, i \mbox{ even}, i \geq 2\\
\Q^{\frac{3i-1}{2}} & \mbox{if } n = i, i \mbox{ odd}, i \geq 3\\
\Q^{\frac{i}{2}} & \mbox{if } n =i-1, i \mbox{ even}, i \geq 4\\
\Q^{\frac{i-3}{2}} & \mbox{if } n = i-1, i \mbox{ odd}, i \geq 3\\
\Q^{2} & \mbox{if } i = 1,\\
\Q & \mbox{if } i =0,\\
0 & \mbox{else}
\end{cases}.
$$
Since $\Conf^n \Sigma_1$ is $K(\pi,1)$ for the elliptic braid group $B_n(\Sigma_1)$ \cite{Birman1969, Scott1970}, this gives
$$
H^{i}(B_n(\Sigma_1); \Q)= H^i(\Conf^n\Sigma_1;\Q).
$$
\end{proposition}

Prior to our work, Napolitano \cite[Table 2]{Napolitano2003} had computed $H_i(\Conf^n\Sigma_1;\Z)$ for $1 \leq n \leq 6$ and $0 \leq i \leq 7$, and Kallel \cite[Corollary 1.7]{Kallel2008} computed $H_1(\Conf^n\Sigma_1; \Z)$ for $n \geq 3.$ Concurrently with our work, Scheissl \cite{Scheissl2016} independently computed the stable and unstable Betti numbers of $\Conf^n\Sigma_1$ also using the Cohen--Taylor--Totaro--Kriz spectral sequence, and Drummond-Cole and Knudsen \cite[Corollaries 4.5--4.7]{Drummond-ColeKnudsen2016} not only computed the stable and unstable Betti numbers of $\Conf^n\Sigma_1$, but did so for all surfaces of finite type via a method derived from factorization homology. For other related computations, see \cite{BrownWhite1981, BodigheimerCohen1988, BodigheimerCohenTaylor1989, Knudsen2015, Azam2015}.

\appendix

\section{Tables of Betti numbers\\ by Matthew Christie and Derek Francour}

\begin{landscape}
\begin{table}
\caption{Betti numbers of $\Conf^n \C\P^1 \times \C\P^1$ for $i \leq 26$ and $n \leq 21$.}
\label{Tab:cp1^2one}
\centering
\begin{tabular}{c|c|c|c|c|c|c|c|c|c|c|c|c|c|c|c|c|c|c|c|c|c|c|c|c|c|c|c}
~ & $b_0$ & $b_1$ & $b_2$ & $b_3$ & $b_4$ & $b_5$ & $b_6$ & $b_7$ & $b_8$ & $b_9$ & $b_{10}$ & $b_{11}$ & $b_{12}$ & $b_{13}$ & $b_{14}$ & $b_{15}$ & $b_{16}$ & $b_{17}$ & $b_{18}$ & $b_{19}$ & $b_{20}$ & $b_{21}$ & $b_{22}$ & $b_{23}$ & $b_{24}$ & $b_{25}$ & $b_{26}$ \\ \hline
$n = 21$ & 1 & 0 & 2 & 0 & 3 & 0 & 2 & 2 & 2 & 4 & 2 & 5 & 2 & 4 & 3 & 4 & 4 & 4 & 4 & 4 & 4 & 4 & 4 & 4 & 4 & 4 & 4 \\
$n = 20$ & 1 & 0 & 2 & 0 & 3 & 0 & 2 & 2 & 2 & 4 & 2 & 5 & 2 & 4 & 3 & 4 & 4 & 4 & 4 & 4 & 4 & 4 & 4 & 4 & 4 & 4 & 4 \\
$n = 19$ & 1 & 0 & 2 & 0 & 3 & 0 & 2 & 2 & 2 & 4 & 2 & 5 & 2 & 4 & 3 & 4 & 4 & 4 & 4 & 4 & 4 & 4 & 4 & 4 & 4 & 4 & 4 \\
$n = 18$ & 1 & 0 & 2 & 0 & 3 & 0 & 2 & 2 & 2 & 4 & 2 & 5 & 2 & 4 & 3 & 4 & 4 & 4 & 4 & 4 & 4 & 4 & 4 & 4 & 4 & 4 & 4 \\
$n = 17$ & 1 & 0 & 2 & 0 & 3 & 0 & 2 & 2 & 2 & 4 & 2 & 5 & 2 & 4 & 3 & 4 & 4 & 4 & 4 & 4 & 4 & 4 & 4 & 4 & 4 & 4 & 4 \\
$n = 16$ & 1 & 0 & 2 & 0 & 3 & 0 & 2 & 2 & 2 & 4 & 2 & 5 & 2 & 4 & 3 & 4 & 4 & 4 & 4 & 4 & 4 & 4 & 4 & 4 & 4 & 4 & 4 \\
$n = 15$ & 1 & 0 & 2 & 0 & 3 & 0 & 2 & 2 & 2 & 4 & 2 & 5 & 2 & 4 & 3 & 4 & 4 & 4 & 4 & 4 & 4 & 4 & 4 & 4 & 4 & 4 & 4 \\
\end{tabular}
\end{table}
\end{landscape}

\begin{landscape}
\begin{table}
\caption{Betti numbers of $\Conf^n \C\P^1 \times \C\P^1$ for $27 \leq i \leq 50$ and $n \leq 24$.}
\label{Tab:cp1^2two}
\centering
\begin{tabular}{c|c|c|c|c|c|c|c|c|c|c|c|c|c|c|c|c|c|c|c|c|c|c|c|c}
~ & $b_{27}$ & $b_{28}$ & $b_{29}$ & $b_{30}$ & $b_{31}$ & $b_{32}$ & $b_{33}$ & $b_{34}$ & $b_{35}$ & $b_{36}$ & $b_{37}$ & $b_{38}$ & $b_{39}$ & $b_{40}$ & $b_{41}$ & $b_{42}$ & $b_{43}$ & $b_{44}$ & $b_{45}$ & $b_{46}$ & $b_{47}$ & $b_{48}$ & $b_{49}$ & $b_{50}$ \\ \hline
$n = 24$ & 4 & 4 & 4 & 4 & 4 & 4 & 4 & 4 & 4 & 4 & 4 & 4 & 4 & 4 & 4 & 4 & 4 & 4 & 4 & 4 & 4 & 4 & 4 & 2 \\
$n = 23$ & 4 & 4 & 4 & 4 & 4 & 4 & 4 & 4 & 4 & 4 & 4 & 4 & 4 & 4 & 4 & 4 & 4 & 4 & 4 & 4 & 4 & 2 & 2 & 2 \\ 
$n = 22$ & 4 & 4 & 4 & 4 & 4 & 4 & 4 & 4 & 4 & 4 & 4 & 4 & 4 & 4 & 4 & 4 & 4 & 4 & 4 & 2 & 2 & 2 & 0 & 0 \\
$n = 21$ & 4 & 4 & 4 & 4 & 4 & 4 & 4 & 4 & 4 & 4 & 4 & 4 & 4 & 4 & 4 & 4 & 4 & 2 & 2 & 2 & 0 & 0 & 0 & 0 \\
$n = 20$ & 4 & 4 & 4 & 4 & 4 & 4 & 4 & 4 & 4 & 4 & 4 & 4 & 4 & 4 & 4 & 2 & 2 & 2 & 0 & 0 & 0 & 0 & 0 & 0 \\
$n = 19$ & 4 & 4 & 4 & 4 & 4 & 4 & 4 & 4 & 4 & 4 & 4 & 4 & 4 & 2 & 2 & 2 & 0 & 0 & 0 & 0 & 0 & 0 & 0 & 0 \\
$n = 18$ & 4 & 4 & 4 & 4 & 4 & 4 & 4 & 4 & 4 & 4 & 4 & 2 & 2 & 2 & 0 & 0 & 0 & 0 & 0 & 0 & 0 & 0 & 0 & 0 \\
\end{tabular}
\end{table}
\end{landscape}

\begin{landscape}
\begin{table}
\caption{Betti numbers of $\Conf^n \C\P^1 \times \C\P^1 \times \C\P^1$ for $i \leq 24$ and $n \leq 11$.}
\label{Tab:cp1^3}
\centering
\begin{tabular}{c|c|c|c|c|c|c|c|c|c|c|c|c|c|c|c|c|c|c|c|c|c|c|c|c|c}
~ & $b_0$ & $b_1$ & $b_2$ & $b_3$ & $b_4$ & $b_5$ & $b_6$ & $b_7$ & $b_8$ & $b_9$ & $b_{10}$ & $b_{11}$ & $b_{12}$ & $b_{13}$ & $b_{14}$ & $b_{15}$ & $b_{16}$ & $b_{17}$ & $b_{18}$ & $b_{19}$ & $b_{20}$ & $b_{21}$ & $b_{22}$ & $b_{23}$ & $b_{24}$ \\ \hline
$n = 11$ & 1 & 0 & 3 & 0 & 9 & 0 & 19 & 0 & 36 & 0 & 57 & 3 & 85 & 15 & 117 & 42 & 156 & 90 & 199 & 162 & 249 & 255 & 309 & 370 & 303 \\
$n = 10$ & 1 & 0 & 3 & 0 & 9 & 0 & 19 & 0 & 36 & 0 & 57 & 3 & 85 & 15 & 117 & 42 & 156 & 90 & 199 & 162 & 249 & 255 & 231 & 370 & 160 \\
$n = 9$ & 1 & 0 & 3 & 0 & 9 & 0 & 19 & 0 & 36 & 0 & 57 & 3 & 85 & 15 & 117 & 42 & 156 & 90 & 199 & 162 & 183 & 255 & 111 & 370 & 78 \\
$n = 8$ & 1 & 0 & 3 & 0 & 9 & 0 & 19 & 0 & 36 & 0 & 57 & 3 & 85 & 15 & 117 & 42 & 156 & 90 & 144 & 162 & 84 & 255 & 42 & 286 & 57 \\
$n = 7$ & 1 & 0 & 3 & 0 & 9 & 0 & 19 & 0 & 36 & 0 & 57 & 3 & 85 & 15 & 117 & 42 & 111 & 90 & 64 & 162 & 27 & 192 & 24 & 135 & 42 \\
$n = 6$ & 1 & 0 & 3 & 0 & 9 & 0 & 19 & 0 & 36 & 0 & 57 & 3 & 85 & 15 & 81 & 42 & 48 & 90 & 18 & 117 & 12 & 84 & 12 & 37 & 24 \\
$n = 5$ & 1 & 0 & 3 & 0 & 9 & 0 & 19 & 0 & 36 & 0 & 57 & 3 & 57 & 15 & 33 & 42 & 12 & 60 & 6 & 45 & 3 & 18 & 3 & 3 & 3 \\
$n = 4$ & 1 & 0 & 3 & 0 & 9 & 0 & 19 & 0 & 36 & 0 & 36 & 3 & 22 & 15 & 6    & 24 & 3 & 17 & 0 & 6 & 0 & 0 & 0 & 0 & 0 \\
$n = 3$ & 1 & 0 & 3 & 0 & 9 & 0 & 19 & 0 & 21 & 0 & 12 & 3 & 3 & 6 & 0 & 3 & 0 & 0 & 0 & 0 & 0 & 0 & 0 & 0 & 0 \\
$n = 2$ & 1 & 0 & 3 & 0 & 9 & 0 & 9 & 0 & 6 & 0 & 0 & 0 & 0 & 0 & 0 & 0 & 0  & 0 & 0 & 0 & 0 & 0 & 0 & 0 & 0 \\
$n = 1$ & 1 & 0 & 3 & 0 & 3 & 0 & 1 & 0 & 0 & 0 & 0 & 0 & 0 & 0 & 0 & 0 & 0 & 0 & 0 & 0 & 0 & 0 & 0 & 0 & 0 
\end{tabular}
\end{table}
\end{landscape}

\begin{landscape}
\begin{table}
\caption{Betti numbers of $\Conf^n \C\P^1 \times \C\P^2$ for $i \leq 26$ and $n \leq 18$.}
\label{Tab:cp1xcp2one}
\centering
\begin{tabular}{c|c|c|c|c|c|c|c|c|c|c|c|c|c|c|c|c|c|c|c|c|c|c|c|c|c|c|c}
~ & $b_0$ & $b_1$ & $b_2$ & $b_3$ & $b_4$ & $b_5$ & $b_6$ & $b_7$ & $b_8$ & $b_9$ & $b_{10}$ & $b_{11}$ & $b_{12}$ & $b_{13}$ & $b_{14}$ & $b_{15}$ & $b_{16}$ & $b_{17}$ & $b_{18}$ & $b_{19}$ & $b_{20}$ & $b_{21}$ & $b_{22}$ & $b_{23}$ & $b_{24}$ & $b_{25}$ & $b_{26}$ \\ \hline
$n = 18$ & 1 & 0 & 2 & 0 & 5 & 0 & 8 & 0 & 12 & 0 & 14 & 2 & 17 & 7 & 19 & 15 & 22 & 24 & 24 & 34 & 27 & 41 & 31 & 49 & 40 & 56 & 49 \\
$n = 17$ & 1 & 0 & 2 & 0 & 5 & 0 & 8 & 0 & 12 & 0 & 14 & 2 & 17 & 7 & 19 & 15 & 22 & 24 & 24 & 34 & 27 & 41 & 31 & 49 & 40 & 56 & 49 \\
$n = 16$ & 1 & 0 & 2 & 0 & 5 & 0 & 8 & 0 & 12 & 0 & 14 & 2 & 17 & 7 & 19 & 15 & 22 & 24 & 24 & 34 & 27 & 41 & 31 & 49 & 40 & 56 & 49 \\
$n = 15$ & 1 & 0 & 2 & 0 & 5 & 0 & 8 & 0 & 12 & 0 & 14 & 2 & 17 & 7 & 19 & 15 & 22 & 24 & 24 & 34 & 27 & 41 & 31 & 49 & 40 & 56 & 49 \\
$n = 14$ & 1 & 0 & 2 & 0 & 5 & 0 & 8 & 0 & 12 & 0 & 14 & 2 & 17 & 7 & 19 & 15 & 22 & 24 & 24 & 34 & 27 & 41 & 31 & 49 & 40 & 56 & 49 \\
$n = 13$ & 1 & 0 & 2 & 0 & 5 & 0 & 8 & 0 & 12 & 0 & 14 & 2 & 17 & 7 & 19 & 15 & 22 & 24 & 24 & 34 & 27 & 41 & 31 & 49 & 40 & 56 & 49 \\
$n = 12$ & 1 & 0 & 2 & 0 & 5 & 0 & 8 & 0 & 12 & 0 & 14 & 2 & 17 & 7 & 19 & 15 & 22 & 24 & 24 & 34 & 27 & 41 & 31 & 49 & 40 & 56 & 35 \\
\end{tabular}
\end{table}
\end{landscape}

\begin{landscape}
\begin{table}
\caption{Betti numbers of $\Conf^n \C\P^1 \times \C\P^2$ for $27 \leq i \leq 50$ and $n \leq 18$.}
\label{Tab:cp1xcp2two}
\centering
\begin{tabular}{c|c|c|c|c|c|c|c|c|c|c|c|c|c|c|c|c|c|c|c|c|c|c|c|c}
~ & $b_{27}$ & $b_{28}$ & $b_{29}$ & $b_{30}$ & $b_{31}$ & $b_{32}$ & $b_{33}$ & $b_{34}$ & $b_{35}$ & $b_{36}$ & $b_{37}$ & $b_{38}$ & $b_{39}$ & $b_{40}$ & $b_{41}$ & $b_{42}$ & $b_{43}$ & $b_{44}$ & $b_{45}$ & $b_{46}$ & $b_{47}$ & $b_{48}$ & $b_{49}$ & $b_{50}$ \\ \hline
$n = 18$ & 64 & 61 & 71 & 71 & 79 & 82 & 87 & 91 & 97 & 102 & 106 & 91 & 117 & 81 & 126 & 85 & 103 & 93 & 58 & 99 & 42 & 93 & 40 & 54 \\
$n = 17$ & 64 & 61 & 71 & 71 & 79 & 82 & 87 & 91 & 97 & 83 & 106 & 72 & 117 & 78 & 94 & 84 & 54 & 92 & 37 & 85 & 37 & 51 & 37 & 24 \\
$n = 16$ & 64 & 61 & 71 & 71 & 79 & 82 & 87 & 73 & 97 & 65 & 106 & 69 & 87 &  77 & 48 & 83 & 34 & 79 & 32 & 46 & 34 & 23 & 34 & 18 \\
$n = 15$ & 64 & 61 & 71 & 71 & 79 & 65 & 87 & 56 & 97 & 62 & 78 & 68 & 44 &  76 & 29 & 71 & 29 & 43 & 29 & 20 & 31 & 17 & 22 & 14 \\
$n = 14$ & 64 & 61 & 71 & 55 & 79 & 49 & 87 & 53 & 71 & 61 & 38 & 67 & 26 & 65 & 24 & 38 & 26 & 19 & 26 & 14 & 20 & 13 & 11 & 10 \\
$n = 13$ & 64 & 46 & 71 & 40 & 79 & 46 & 63 & 52 & 34 & 60 & 21 & 57 & 21 & 35 & 21 & 16 & 23 & 13 & 16 & 10 & 10 & 9 & 7 & 6 \\
$n = 12$ & 64 & 32 & 71 & 37 & 57 & 45 & 29 & 51 & 18 & 51 & 16 & 30 & 18 & 15 & 18 & 10 & 14 & 9 & 7 & 6 & 6 & 5 & 3 & 2 \\
\end{tabular}
\end{table}
\end{landscape}

\begin{landscape}
\begin{table}
\caption{Betti numbers of $\Conf^n \C\P^1 \times \C\P^2$ for $51 \leq i \leq 74$ and $n \leq 18$.}
\label{Tab:cp1xcp2three}
\centering
\begin{tabular}{c|c|c|c|c|c|c|c|c|c|c|c|c|c|c|c|c|c|c|c|c|c|c|c|c}
~ & $b_{51}$ & $b_{52}$ & $b_{53}$ & $b_{54}$ & $b_{55}$ & $b_{56}$ & $b_{57}$ & $b_{58}$ & $b_{59}$ & $b_{60}$ & $b_{61}$ & $b_{62}$ & $b_{63}$ & $b_{64}$ & $b_{65}$ & $b_{66}$ & $b_{67}$ & $b_{68}$ & $b_{69}$ & $b_{70}$ & $b_{71}$ & $b_{72}$ & $b_{73}$ & $b_{74}$ \\ \hline
$n = 18$ & 42 & 27 & 42 & 22 & 32 & 21 & 19 & 18 & 18 & 17 & 15 & 14 & 14 & 13 & 11 & 10 & 10 & 9 & 7 & 6 & 6 & 5 & 3 & 2 \\
$n = 17$ & 39 & 21 & 28 & 18 & 18 & 17 & 15 & 14 & 14 & 13 & 11 & 10 & 10 & 9 & 7 & 6 & 6 & 5 & 3 & 2 & 2 & 1 & 0 & 0 \\ 
$n = 16$ & 26 & 17 & 15 & 14 & 14 & 13 & 11 & 10 & 10 & 9 & 7 & 6 & 6 & 5 & 3 & 2 & 2 & 1 & 0 & 0 & 0 & 0 & 0 & 0 \\
$n = 15$ & 14 & 13 & 11 & 10 & 10 & 9 & 7 & 6 & 6 & 5 & 3 & 2 & 2 & 1 & 0 & 0 & 0 & 0 & 0 & 0 & 0 & 0 & 0 & 0 \\
$n = 14$ & 10 & 9 & 7 & 6 & 6 & 5 & 3 & 2 & 2 & 1 & 0 & 0 & 0 & 0 & 0 & 0 & 0 & 0 & 0 & 0 & 0 & 0 & 0 & 0 \\
$n = 13$ & 6 & 5 & 3 & 2 & 2 & 1 & 0 & 0 & 0 & 0 & 0 & 0 & 0 & 0 & 0 & 0 & 0 & 0 & 0 & 0 & 0 & 0 & 0 & 0 \\
$n = 12$ & 2 & 1 & 0 & 0 & 0 & 0 & 0 & 0 & 0 & 0 & 0 & 0 & 0 & 0 & 0 & 0 & 0 & 0 & 0 & 0 & 0 & 0 & 0 & 0 \\
\end{tabular}
\end{table}
\end{landscape}

\begin{landscape}
\begin{table}
\caption{Betti numbers of $\Conf^n \P_{\C\P^2}(\O \oplus \O(1))$ for $i \leq 26$ and $n \leq 18$.}
\label{Tab:cp1xcp2Tone}
\centering
\begin{tabular}{c|c|c|c|c|c|c|c|c|c|c|c|c|c|c|c|c|c|c|c|c|c|c|c|c|c|c|c}
~ & $b_0$ & $b_1$ & $b_2$ & $b_3$ & $b_4$ & $b_5$ & $b_6$ & $b_7$ & $b_8$ & $b_9$ & $b_{10}$ & $b_{11}$ & $b_{12}$ & $b_{13}$ & $b_{14}$ & $b_{15}$ & $b_{16}$ & $b_{17}$ & $b_{18}$ & $b_{19}$ & $b_{20}$ & $b_{21}$ & $b_{22}$ & $b_{23}$ & $b_{24}$ & $b_{25}$ & $b_{26}$ \\ \hline
$n = 18$ & 1 & 0 & 2 & 0 & 5 & 0 & 8 & 0 & 12 & 0 & 14 & 1 & 16 & 6 & 18 & 13 & 20 & 22 & 22 & 31 & 24 & 38 & 26 & 43 & 33 & 49 & 41 \\
$n = 17$ & 1 & 0 & 2 & 0 & 5 & 0 & 8 & 0 & 12 & 0 & 14 & 1 & 16 & 6 & 18 & 13 & 20 & 22 & 22 & 31 & 24 & 38 & 26 & 43 & 33 & 49 & 41 \\
$n = 16$ & 1 & 0 & 2 & 0 & 5 & 0 & 8 & 0 & 12 & 0 & 14 & 1 & 16 & 6 & 18 & 13 & 20 & 22 & 22 & 31 & 24 & 38 & 26 & 43 & 33 & 49 & 41 \\
$n = 15$ & 1 & 0 & 2 & 0 & 5 & 0 & 8 & 0 & 12 & 0 & 14 & 1 & 16 & 6 & 18 & 13 & 20 & 22 & 22 & 31 & 24 & 38 & 26 & 43 & 33 & 49 & 41 \\
$n = 14$ & 1 & 0 & 2 & 0 & 5 & 0 & 8 & 0 & 12 & 0 & 14 & 1 & 16 & 6 & 18 & 13 & 20 & 22 & 22 & 31 & 24 & 38 & 26 & 43 & 33 & 49 & 41 \\
$n = 13$ & 1 & 0 & 2 & 0 & 5 & 0 & 8 & 0 & 12 & 0 & 14 & 1 & 16 & 6 & 18 & 13 & 20 & 22 & 22 & 31 & 24 & 38 & 26 & 43 & 33 & 49 & 41 \\
$n = 12$ & 1 & 0 & 2 & 0 & 5 & 0 & 8 & 0 & 12 & 0 & 14 & 1 & 16 & 6 & 18 & 13 & 20 & 22 & 22 & 31 & 24 & 38 & 26 & 43 & 33 & 49 & 27 \\
\end{tabular}
\end{table}
\end{landscape}

\begin{landscape}
\begin{table}
\caption{Betti numbers of $\Conf^n \P_{\C\P^2}(\O \oplus \O(1))$ for $27 \leq i \leq 51$ and $n \leq 18$.}
\label{Tab:cp1xcp2Ttwo}
\centering
\begin{tabular}{c|c|c|c|c|c|c|c|c|c|c|c|c|c|c|c|c|c|c|c|c|c|c|c|c|c}
~ & $b_{27}$ & $b_{28}$ & $b_{29}$ & $b_{30}$ & $b_{31}$ & $b_{32}$ & $b_{33}$ & $b_{34}$ & $b_{35}$ & $b_{36}$ & $b_{37}$ & $b_{38}$ & $b_{39}$ & $b_{40}$ & $b_{41}$ & $b_{42}$ & $b_{43}$ & $b_{44}$ & $b_{45}$ & $b_{46}$ & $b_{47}$ & $b_{48}$ & $b_{49}$ & $b_{50}$ & $b_{51}$\\ \hline
$n = 18$ & 55 & 51 & 61 & 60 & 67 & 69 & 73 & 76 & 80 & 84 & 88 & 72 & 96 & 59 & 104 & 62 & 79 & 68 & 34 & 74 & 16 & 67 & 15 & 30 & 17 \\
$n = 17$ & 55 & 51 & 61 & 60 & 67 & 69 & 73 & 76 & 80 & 84 & 88 & 72 & 96 & 59 & 104 & 62 & 79 & 68 & 34 & 74 & 16 & 67 & 15 & 30 & 17 \\
$n = 16$ & 55 & 51 & 61 & 60 & 67 & 69 & 73 & 58 & 80 & 47 & 88 & 50 & 67 & 56 & 28 & 62 & 12 & 57 & 11 & 26 & 13 & 3 & 15 & 0 & 8 \\
$n = 15$ & 55 & 51 & 61 & 60 & 67 & 52 & 73 & 41 & 80 & 44 & 61 & 50 & 25 & 56 & 10 & 52 & 9 & 24 & 11 & 3 & 13 & 0 & 7 & 0 & 0 \\
$n = 14$ & 55 & 51 & 61 & 44 & 67 & 36 & 73 & 38 & 55 & 44 & 22 & 50 & 8 & 47 & 7 & 22 & 9 & 3 & 11 & 0 & 6 & 0 & 0 & 0 & 0 \\
$n = 13$ & 55 & 36 & 61 & 29 & 67 & 33 & 50 & 38 & 19 & 44 & 6 & 42 & 5 & 20 & 7 & 3 & 9 & 0 & 5 & 0 & 0 & 0 & 0 & 0 & 0 \\
$n = 12$ & 55 & 22 & 61 & 26 & 46 & 33 & 17 & 38 & 4 & 37 & 3 & 18 & 5 & 3 & 7 & 0 & 4 & 0 & 0 & 0 & 0 & 0 & 0 & 0 & 0 \\
\end{tabular}
\end{table}
\end{landscape}

\begin{landscape}
\begin{table}
\caption{Betti numbers of $\Conf^n \C\P^3$ for $i \leq 26$ and $n \leq 21$.}
\label{Tab:cp3one}
\centering
\begin{tabular}{c|c|c|c|c|c|c|c|c|c|c|c|c|c|c|c|c|c|c|c|c|c|c|c|c|c|c|c}
~ & $b_0$ & $b_1$ & $b_2$ & $b_3$ & $b_4$ & $b_5$ & $b_6$ & $b_7$ & $b_8$ & $b_9$ & $b_{10}$ & $b_{11}$ & $b_{12}$ & $b_{13}$ & $b_{14}$ & $b_{15}$ & $b_{16}$ & $b_{17}$ & $b_{18}$ & $b_{19}$ & $b_{20}$ & $b_{21}$ & $b_{22}$ & $b_{23}$ & $b_{24}$ & $b_{25}$ & $b_{26}$ \\ \hline
$n = 21$ & 1 & 0 & 1 & 0 & 2 & 0 & 2 & 0 & 2 & 0 & 1 & 1 & 1 & 2 & 1 & 3 & 1 & 3 & 1 & 3 & 1 & 2 & 1 & 2 & 2 & 2 & 2 \\
$n = 20$ & 1 & 0 & 1 & 0 & 2 & 0 & 2 & 0 & 2 & 0 & 1 & 1 & 1 & 2 & 1 & 3 & 1 & 3 & 1 & 3 & 1 & 2 & 1 & 2 & 2 & 2 & 2 \\
$n = 19$ & 1 & 0 & 1 & 0 & 2 & 0 & 2 & 0 & 2 & 0 & 1 & 1 & 1 & 2 & 1 & 3 & 1 & 3 & 1 & 3 & 1 & 2 & 1 & 2 & 2 & 2 & 2 \\
$n = 18$ & 1 & 0 & 1 & 0 & 2 & 0 & 2 & 0 & 2 & 0 & 1 & 1 & 1 & 2 & 1 & 3 & 1 & 3 & 1 & 3 & 1 & 2 & 1 & 2 & 2 & 2 & 2 \\
$n = 17$ & 1 & 0 & 1 & 0 & 2 & 0 & 2 & 0 & 2 & 0 & 1 & 1 & 1 & 2 & 1 & 3 & 1 & 3 & 1 & 3 & 1 & 2 & 1 & 2 & 2 & 2 & 2 \\
$n = 16$ & 1 & 0 & 1 & 0 & 2 & 0 & 2 & 0 & 2 & 0 & 1 & 1 & 1 & 2 & 1 & 3 & 1 & 3 & 1 & 3 & 1 & 2 & 1 & 2 & 2 & 2 & 2 \\
$n = 15$ & 1 & 0 & 1 & 0 & 2 & 0 & 2 & 0 & 2 & 0 & 1 & 1 & 1 & 2 & 1 & 3 & 1 & 3 & 1 & 3 & 1 & 2 & 1 & 2 & 2 & 2 & 2 \\
\end{tabular}
\end{table}
\end{landscape}

\begin{landscape}
\begin{table}
\caption{Betti numbers of $\Conf^n \C\P^3$ for $27 \leq i \leq 50$ and $n \leq 21$.}
\label{Tab:cp3two}
\centering
\begin{tabular}{c|c|c|c|c|c|c|c|c|c|c|c|c|c|c|c|c|c|c|c|c|c|c|c|c}
~ & $b_{27}$ & $b_{28}$ & $b_{29}$ & $b_{30}$ & $b_{31}$ & $b_{32}$ & $b_{33}$ & $b_{34}$ & $b_{35}$ & $b_{36}$ & $b_{37}$ & $b_{38}$ & $b_{39}$ & $b_{40}$ & $b_{41}$ & $b_{42}$ & $b_{43}$ & $b_{44}$ & $b_{45}$ & $b_{46}$ & $b_{47}$ & $b_{48}$ & $b_{49}$ & $b_{50}$ \\ \hline
$n = 21$ & 2 & 2 & 2 & 2 & 2 & 2 & 2 & 2 & 2 & 2 & 2 & 2 & 2 & 2 & 2 & 2 & 2 & 1 & 2 & 1 & 2 & 1 & 1 & 1 \\
$n = 20$ & 2 & 2 & 2 & 2 & 2 & 2 & 2 & 2 & 2 & 2 & 2 & 2 & 2 & 2 & 2 & 1 & 2 & 1 & 2 & 1 & 1 & 1 & 0 & 1 \\
$n = 19$ & 2 & 2 & 2 & 2 & 2 & 2 & 2 & 2 & 2 & 2 & 2 & 2 & 2 & 1 & 2 & 1 & 2 & 1 & 1 & 1 & 0 & 1 & 0 & 1 \\
$n = 18$ & 2 & 2 & 2 & 2 & 2 & 2 & 2 & 2 & 2 & 2 & 2 & 1 & 2 & 1 & 2 & 1 & 1 & 1 & 0 & 1 & 0 & 1 & 0 & 0 \\
$n = 17$ & 2 & 2 & 2 & 2 & 2 & 2 & 2 & 2 & 2 & 1 & 2 & 1 & 2 & 1 & 1 & 1 & 0 & 1 & 0 & 1 & 0 & 0 & 0 & 0 \\
$n = 16$ & 2 & 2 & 2 & 2 & 2 & 2 & 2 & 1 & 2 & 1 & 2 & 1 & 1 & 1 & 0 & 1 & 0 & 1 & 0 & 0 & 0 & 0 & 0 & 0 \\
$n = 15$ & 2 & 2 & 2 & 2 & 2 & 1 & 2 & 1 & 2 & 1 & 1 & 1 & 0 & 1 & 0 & 1 & 0 & 0 & 0 & 0 & 0 & 0 & 0 & 0 \\
\end{tabular}
\end{table}
\end{landscape}

\begin{landscape}
\begin{table}
\caption{Betti numbers of $\Conf^n \C\P^4$ for $i \leq 26$ and $n \leq 21$.}
\label{Tab:cp4one}
\centering
\begin{tabular}{c|c|c|c|c|c|c|c|c|c|c|c|c|c|c|c|c|c|c|c|c|c|c|c|c|c|c|c}
~ & $b_0$ & $b_1$ & $b_2$ & $b_3$ & $b_4$ & $b_5$ & $b_6$ & $b_7$ & $b_8$ & $b_9$ & $b_{10}$ & $b_{11}$ & $b_{12}$ & $b_{13}$ & $b_{14}$ & $b_{15}$ & $b_{16}$ & $b_{17}$ & $b_{18}$ & $b_{19}$ & $b_{20}$ & $b_{21}$ & $b_{22}$ & $b_{23}$ & $b_{24}$ & $b_{25}$ & $b_{26}$ \\ \hline
$n = 21$ & 1 & 0 & 1 & 0 & 2 & 0 & 3 & 0 & 4 & 0 & 4 & 0 & 5 & 0 & 4 & 1 & 4 & 2 & 4 & 4 & 4 & 6 & 4 & 7 & 4 & 8 & 4 \\
$n = 20$ & 1 & 0 & 1 & 0 & 2 & 0 & 3 & 0 & 4 & 0 & 4 & 0 & 5 & 0 & 4 & 1 & 4 & 2 & 4 & 4 & 4 & 6 & 4 & 7 & 4 & 8 & 4 \\
$n = 19$ & 1 & 0 & 1 & 0 & 2 & 0 & 3 & 0 & 4 & 0 & 4 & 0 & 5 & 0 & 4 & 1 & 4 & 2 & 4 & 4 & 4 & 6 & 4 & 7 & 4 & 8 & 4 \\
$n = 18$ & 1 & 0 & 1 & 0 & 2 & 0 & 3 & 0 & 4 & 0 & 4 & 0 & 5 & 0 & 4 & 1 & 4 & 2 & 4 & 4 & 4 & 6 & 4 & 7 & 4 & 8 & 4 \\
$n = 17$ & 1 & 0 & 1 & 0 & 2 & 0 & 3 & 0 & 4 & 0 & 4 & 0 & 5 & 0 & 4 & 1 & 4 & 2 & 4 & 4 & 4 & 6 & 4 & 7 & 4 & 8 & 4 \\
$n = 16$ & 1 & 0 & 1 & 0 & 2 & 0 & 3 & 0 & 4 & 0 & 4 & 0 & 5 & 0 & 4 & 1 & 4 & 2 & 4 & 4 & 4 & 6 & 4 & 7 & 4 & 8 & 4 \\
$n = 15$ & 1 & 0 & 1 & 0 & 2 & 0 & 3 & 0 & 4 & 0 & 4 & 0 & 5 & 0 & 4 & 1 & 4 & 2 & 4 & 4 & 4 & 6 & 4 & 7 & 4 & 8 & 4 \\
\end{tabular}
\end{table}
\end{landscape}

\begin{landscape}
\begin{table}
\caption{Betti numbers of $\Conf^n \C\P^4$ for $27 \leq i \leq 47$ and $n \leq 23$.}
\label{Tab:cp4two}
\centering
\begin{tabular}{c|c|c|c|c|c|c|c|c|c|c|c|c|c|c|c|c|c|c|c|c|c}
~ & $b_{27}$ & $b_{28}$ & $b_{29}$ & $b_{30}$ & $b_{31}$ & $b_{32}$ & $b_{33}$ & $b_{34}$ & $b_{35}$ & $b_{36}$ & $b_{37}$ & $b_{38}$ & $b_{39}$ & $b_{40}$ & $b_{41}$ & $b_{42}$ & $b_{43}$ & $b_{44}$ & $b_{45}$ & $b_{46}$ & $b_{47}$ \\ \hline
$n = 23$ & 9 & 4 & 8 & 4 & 8 & 5 & 8 & 6 & 8 & 7 & 8 & 7 & 8 & 8 & 8 & 8 & 8 & 8 & 8 & 8 & 8\\
$n = 22$ & 9 & 4 & 8 & 4 & 8 & 5 & 8 & 6 & 8 & 7 & 8 & 7 & 8 & 8 & 8 & 8 & 8 & 8 & 8 & 7 & 8\\
$n = 21$ & 9 & 4 & 8 & 4 & 8 & 5 & 8 & 6 & 8 & 7 & 8 & 7 & 8 & 8 & 8 & 8 & 8 & 7 & 8 & 6 & 8\\
$n = 20$ & 9 & 4 & 8 & 4 & 8 & 5 & 8 & 6 & 8 & 7 & 8 & 7 & 8 & 8 & 8 & 7 & 8 & 6 & 8 & 5 & 8\\
$n = 19$ & 9 & 4 & 8 & 4 & 8 & 5 & 8 & 6 & 8 & 7 & 8 & 7 & 8 & 7 & 8 & 6 & 8 & 5 & 8 & 4 & 8\\
$n = 18$ & 9 & 4 & 8 & 4 & 8 & 5 & 8 & 6 & 8 & 7 & 8 & 6 & 8 & 6 & 8 & 5 & 8 & 4 & 8 & 4 & 7\\
$n = 17$ & 9 & 4 & 8 & 4 & 8 & 5 & 8 & 6 & 8 & 6 & 8 & 5 & 8 & 5 & 8 & 4 & 8 & 4 & 7 & 4 & 5\\
\end{tabular}
\end{table}
\end{landscape}

\begin{landscape}
\begin{table}
\caption{Betti numbers of $\Conf^n \C\P^4$ for $48 \leq i \leq 71$ and $n \leq 30$.}
\label{Tab:cp4three}
\centering
\begin{tabular}{c|c|c|c|c|c|c|c|c|c|c|c|c|c|c|c|c|c|c|c|c|c|c|c|c}
~ & $b_{48}$ & $b_{49}$ & $b_{50}$ & $b_{51}$ & $b_{52}$ & $b_{53}$ & $b_{54}$ & $b_{55}$ & $b_{56}$ & $b_{57}$ & $b_{58}$ & $b_{59}$ & $b_{60}$ & $b_{61}$ & $b_{62}$ & $b_{63}$ & $b_{64}$ & $b_{65}$ & $b_{66}$ & $b_{67}$ & $b_{68}$ & $b_{69}$ & $b_{70}$ & $b_{71}$ \\ \hline
$n = 30$ & 8 & 8 & 8 & 8 & 8 & 8 & 8 & 8 & 8 & 8 & 8 & 8 & 8 & 8 & 7 & 8 & 6 & 8 & 5 & 8 & 4 & 8 & 4 & 7\\
$n = 29$ & 8 & 8 & 8 & 8 & 8 & 8 & 8 & 8 & 8 & 8 & 8 & 8 & 7 & 8 & 6 & 8 & 5 & 8 & 4 & 8 & 4 & 7 & 4 & 5\\
$n = 28$ & 8 & 8 & 8 & 8 & 8 & 8 & 8 & 8 & 8 & 8 & 7 & 8 & 6 & 8 & 5 & 8 & 4 & 8 & 4 & 7 & 4 & 5 & 4 & 3\\
$n = 27$ & 8 & 8 & 8 & 8 & 8 & 8 & 8 & 8 & 7 & 8 & 6 & 8 & 5 & 8 & 4 & 8 & 4 & 7 & 4 & 5 & 4 & 3 & 4 & 1\\
$n = 26$ & 8 & 8 & 8 & 8 & 8 & 8 & 7 & 8 & 6 & 8 & 5 & 8 & 4 & 8 & 4 & 7 & 4 & 5 & 4 & 3 & 4 & 1 & 4 & 0\\
$n = 25$ & 8 & 8 & 8 & 8 & 7 & 8 & 6 & 8 & 5 & 8 & 4 & 8 & 4 & 7 & 4 & 5 & 4 & 3 & 4 & 1 & 4 & 0 & 4 & 0\\
$n = 24$ & 8 & 8 & 7 & 8 & 6 & 8 & 5 & 8 & 4 & 8 & 4 & 7 & 4 & 5 & 4 & 3 & 4 & 1 & 4 & 0 & 4 & 0 & 3 & 0\\
\end{tabular}
\end{table}
\end{landscape}

\begin{landscape}
\begin{table}
\caption{Betti numbers of $\Conf^n \C\P^5$ for $i \leq 26$ and $n \leq 21$.}
\label{Tab:cp5one}
\centering
\begin{tabular}{c|c|c|c|c|c|c|c|c|c|c|c|c|c|c|c|c|c|c|c|c|c|c|c|c|c|c|c}
~ & $b_0$ & $b_1$ & $b_2$ & $b_3$ & $b_4$ & $b_5$ & $b_6$ & $b_7$ & $b_8$ & $b_9$ & $b_{10}$ & $b_{11}$ & $b_{12}$ & $b_{13}$ & $b_{14}$ & $b_{15}$ & $b_{16}$ & $b_{17}$ & $b_{18}$ & $b_{19}$ & $b_{20}$ & $b_{21}$ & $b_{22}$ & $b_{23}$ & $b_{24}$ & $b_{25}$ & $b_{26}$ \\ \hline
$n = 21$ & 1 & 0 & 1 & 0 & 2 & 0 & 3 & 0 & 5 & 0 & 6 & 0 & 8 & 0 & 9 & 0 & 11 & 0 & 11 & 1 & 12 & 2 & 12 & 4 & 13 & 7 & 13\\
$n = 20$ & 1 & 0 & 1 & 0 & 2 & 0 & 3 & 0 & 5 & 0 & 6 & 0 & 8 & 0 & 9 & 0 & 11 & 0 & 11 & 1 & 12 & 2 & 12 & 4 & 13 & 7 & 13\\
$n = 19$ & 1 & 0 & 1 & 0 & 2 & 0 & 3 & 0 & 5 & 0 & 6 & 0 & 8 & 0 & 9 & 0 & 11 & 0 & 11 & 1 & 12 & 2 & 12 & 4 & 13 & 7 & 13\\
$n = 18$ & 1 & 0 & 1 & 0 & 2 & 0 & 3 & 0 & 5 & 0 & 6 & 0 & 8 & 0 & 9 & 0 & 11 & 0 & 11 & 1 & 12 & 2 & 12 & 4 & 13 & 7 & 13\\
$n = 17$ & 1 & 0 & 1 & 0 & 2 & 0 & 3 & 0 & 5 & 0 & 6 & 0 & 8 & 0 & 9 & 0 & 11 & 0 & 11 & 1 & 12 & 2 & 12 & 4 & 13 & 7 & 13\\
$n = 16$ & 1 & 0 & 1 & 0 & 2 & 0 & 3 & 0 & 5 & 0 & 6 & 0 & 8 & 0 & 9 & 0 & 11 & 0 & 11 & 1 & 12 & 2 & 12 & 4 & 13 & 7 & 13\\
$n = 15$ & 1 & 0 & 1 & 0 & 2 & 0 & 3 & 0 & 5 & 0 & 6 & 0 & 8 & 0 & 9 & 0 & 11 & 0 & 11 & 1 & 12 & 2 & 12 & 4 & 13 & 7 & 13\\
\end{tabular}
\end{table}
\end{landscape}

\begin{landscape}
\begin{table}
\caption{Betti numbers of $\Conf^n \C\P^5$ for $27 \leq i \leq 50$ and $n \leq 22$.}
\label{Tab:cp5two}
\centering
\begin{tabular}{c|c|c|c|c|c|c|c|c|c|c|c|c|c|c|c|c|c|c|c|c|c|c|c|c}
~ & $b_{27}$ & $b_{28}$ & $b_{29}$ & $b_{30}$ & $b_{31}$ & $b_{32}$ & $b_{33}$ & $b_{34}$ & $b_{35}$ & $b_{36}$ & $b_{37}$ & $b_{38}$ & $b_{39}$ & $b_{40}$ & $b_{41}$ & $b_{42}$ & $b_{43}$ & $b_{44}$ & $b_{45}$ & $b_{46}$ & $b_{47}$ & $b_{48}$ & $b_{49}$ & $b_{50}$ \\ \hline
$n = 22$ & 11 & 14 & 14 & 14 & 18 & 15 & 21 & 15 & 25 & 16 & 26 & 16 & 28 & 18 & 29 & 19 & 31 & 23 & 32 & 24 & 34 & 27 & 35 & 27\\
$n = 21$ & 11 & 14 & 14 & 14 & 18 & 15 & 21 & 15 & 25 & 16 & 26 & 16 & 28 & 18 & 29 & 19 & 31 & 22 & 32 & 23 & 34 & 25 & 35 & 25\\
$n = 20$ & 11 & 14 & 14 & 14 & 18 & 15 & 21 & 15 & 25 & 16 & 26 & 16 & 28 & 18 & 29 & 18 & 31 & 21 & 32 & 21 & 34 & 23 & 35 & 22\\
$n = 19$ & 11 & 14 & 14 & 14 & 18 & 15 & 21 & 15 & 25 & 16 & 26 & 16 & 28 & 17 & 29 & 17 & 31 & 19 & 32 & 19 & 34 & 20 & 35 & 20\\
$n = 18$ & 11 & 14 & 14 & 14 & 18 & 15 & 21 & 15 & 25 & 16 & 26 & 15 & 28 & 16 & 29 & 15 & 31 & 17 & 32 & 16 & 34 & 18 & 35 & 18\\
$n = 17$ & 11 & 14 & 14 & 14 & 18 & 15 & 21 & 15 & 25 & 15 & 26 & 14 & 28 & 14 & 29 & 13 & 31 & 14 & 32 & 14 & 34 & 16 & 34 & 17\\
$n = 16$ & 11 & 14 & 14 & 14 & 18 & 15 & 21 & 14 & 25 & 14 & 26 & 12 & 28 & 12 & 29 & 10 & 31 & 12 & 32 & 12 & 33 & 15 & 32 & 16\\
\end{tabular}
\end{table}
\end{landscape}

\begin{landscape}
\begin{table}
\caption{Betti numbers of $\Conf^n \C\P^5$ for $51 \leq i \leq 74$ and $n \leq 22$.}
\label{Tab:cp5three}
\centering
\begin{tabular}{c|c|c|c|c|c|c|c|c|c|c|c|c|c|c|c|c|c|c|c|c|c|c|c|c}
~ & $b_{51}$ & $b_{52}$ & $b_{53}$ & $b_{54}$ & $b_{55}$ & $b_{56}$ & $b_{57}$ & $b_{58}$ & $b_{59}$ & $b_{60}$ & $b_{61}$ & $b_{62}$ & $b_{63}$ & $b_{64}$ & $b_{65}$ & $b_{66}$ & $b_{67}$ & $b_{68}$ & $b_{69}$ & $b_{70}$ & $b_{71}$ & $b_{72}$ & $b_{73}$ & $b_{74}$ \\ \hline
$n = 22$ & 37 & 29 & 38 & 28 & 40 & 29 & 41 & 28 & 42 & 30 & 41 & 30 & 40 & 32 & 36 & 32 & 33 & 34 & 28 & 34 & 26 & 36 & 23 & 35\\
$n = 21$ & 37 & 26 & 38 & 26 & 40 & 27 & 40 & 27 & 40 & 29 & 37 & 29 & 35 & 31 & 30 & 31 & 27 & 33 & 23 & 33 & 22 & 34 & 20 & 32\\
$n = 20$ & 37 & 24 & 38 & 24 & 39 & 26 & 38 & 26 & 36 & 28 & 32 & 28 & 29 & 30 & 24 & 30 & 22 & 32 & 19 & 31 & 19 & 31 & 17 & 27\\
$n = 19$ & 37 & 22 & 37 & 23 & 37 & 25 & 34 & 25 & 31 & 27 & 26 & 27 & 23 & 29 & 19 & 29 & 18 & 30 & 16 & 28 & 16 & 26 & 14 & 22\\
$n = 18$ & 36 & 21 & 35 & 22 & 33 & 24 & 29 & 24 & 25 & 26 & 20 & 26 & 18 & 28 & 15 & 27 & 15 & 27 & 13 & 23 & 13 & 21 & 11 & 16\\
$n = 17$ & 34 & 20 & 31 & 21 & 28 & 23 & 23 & 23 & 19 & 25 & 15 & 25 & 14 & 26 & 12 & 24 & 12 & 22 & 10 & 18 & 10 & 15 & 8 & 12\\
$n = 16$ & 30 & 19 & 26 & 20 & 22 & 22 & 17 & 22 & 14 & 24 & 11 & 23 & 11 & 23 & 9 & 19 & 9 & 17 & 7 & 13 & 7 & 11 & 5 & 8\\
\end{tabular}
\end{table}
\end{landscape}

\begin{landscape}
\begin{table}
\caption{Betti numbers of $\Conf^n \C\P^5$ for $75 \leq i \leq 98$ and $n \leq 22$.}
\label{Tab:cp5four}
\centering
\begin{tabular}{c|c|c|c|c|c|c|c|c|c|c|c|c|c|c|c|c|c|c|c|c|c|c|c|c}
~ & $b_{75}$ & $b_{76}$ & $b_{77}$ & $b_{78}$ & $b_{79}$ & $b_{80}$ & $b_{81}$ & $b_{82}$ & $b_{83}$ & $b_{84}$ & $b_{85}$ & $b_{86}$ & $b_{87}$ & $b_{88}$ & $b_{89}$ & $b_{90}$ & $b_{91}$ & $b_{92}$ & $b_{93}$ & $b_{94}$ & $b_{95}$ & $b_{96}$ & $b_{97}$ & $b_{98}$ \\ \hline
$n = 22$ & 23 & 35 & 21 & 31 & 21 & 29 & 19 & 24 & 19 & 22 & 17 & 18 & 17 & 17 & 15 & 14 & 14 & 13 & 11 & 11 & 10 & 10 & 7 & 8\\
$n = 21$ & 20 & 30 & 18 & 26 & 18 & 23 & 16 & 19 & 16 & 17 & 14 & 14 & 14 & 13 & 11 & 11 & 10 & 10 & 7 & 8 & 6 & 7 & 4 & 5\\
$n = 20$ & 17 & 25 & 15 & 20 & 15 & 18 & 13 & 14 & 13 & 13 & 11 & 11 & 10 & 10 & 7 & 8 & 6 & 7 & 4 & 5 & 4 & 4 & 3 & 2\\
$n = 19$ & 14 & 19 & 12 & 15 & 12 & 13 & 10 & 11 & 10 & 10 & 7 & 8 & 6 & 7 & 4 & 5 & 4 & 4 & 3 & 2 & 3 & 1 & 2 & 0\\
$n = 18$ & 11 & 14 & 9 & 11 & 9 & 10 & 7 & 8 & 6 & 7 & 4 & 5 & 4 & 4 & 3 & 2 & 3 & 1 & 2 & 0 & 2 & 0 & 1 & 0\\
$n = 17$ & 8 & 10 & 6 & 8 & 6 & 7 & 4 & 5 & 4 & 4 & 3 & 2 & 3 & 1 & 2 & 0 & 2 & 0 & 1 & 0 & 1 & 0 & 0 & 0\\
$n = 16$ & 5 & 7 & 4 & 5 & 4 & 4 & 3 & 2 & 3 & 1 & 2 & 0 & 2 & 0 & 1 & 0 & 1 & 0 & 0 & 0 & 0 & 0 & 0 & 0\\
$n = 15$ & 4 & 4 & 3 & 2 & 3 & 1 & 2 & 0 & 2 & 0 & 1 & 0 & 1 & 0 & 0 & 0 & 0 & 0 & 0 & 0 & 0 & 0 & 0 & 0\\
$n = 14$ & 3 & 1 & 2 & 0 & 2 & 0 & 1 & 0 & 1 & 0 & 0 & 0 & 0 & 0 & 0 & 0 & 0 & 0 & 0 & 0 & 0 & 0 & 0 & 0\\
$n = 13$ & 2 & 0 & 1 & 0 & 1 & 0 & 0 & 0 & 0 & 0 & 0 & 0 & 0 & 0 & 0 & 0 & 0 & 0 & 0 & 0 & 0 & 0 & 0 & 0\\
$n = 12$ & 1 & 0 & 0 & 0 & 0 & 0 & 0 & 0 & 0 & 0 & 0 & 0 & 0 & 0 & 0 & 0 & 0 & 0 & 0 & 0 & 0 & 0 & 0 & 0\\
$n \leq 11$ & 0 & 0 & 0 & 0 & 0 & 0 & 0 & 0 & 0 & 0 & 0 & 0 & 0 & 0 & 0 & 0 & 0 & 0 & 0 & 0 & 0 & 0 & 0 & 0
\end{tabular}
\end{table}
\end{landscape}

\begin{landscape}
\begin{table}
\caption{Betti numbers of $\Conf^n \C\P^6$ for $i \leq 26$ and $n \leq 17$.}
\label{Tab:cp6one}
\centering
\begin{tabular}{c|c|c|c|c|c|c|c|c|c|c|c|c|c|c|c|c|c|c|c|c|c|c|c|c|c|c|c}
~ & $b_0$ & $b_1$ & $b_2$ & $b_3$ & $b_4$ & $b_5$ & $b_6$ & $b_7$ & $b_8$ & $b_9$ & $b_{10}$ & $b_{11}$ & $b_{12}$ & $b_{13}$ & $b_{14}$ & $b_{15}$ & $b_{16}$ & $b_{17}$ & $b_{18}$ & $b_{19}$ & $b_{20}$ & $b_{21}$ & $b_{22}$ & $b_{23}$ & $b_{24}$ & $b_{25}$ & $b_{26}$ \\ \hline
$n = 17$ & 1 & 0 & 1 & 0 & 2 & 0 & 3 & 0 & 5 & 0 & 7 & 0 & 10 & 0 & 12 & 0 & 16 & 0 & 19 & 0 & 23 & 0 & 25 & 0 & 28 & 1 & 30\\
$n = 16$ & 1 & 0 & 1 & 0 & 2 & 0 & 3 & 0 & 5 & 0 & 7 & 0 & 10 & 0 & 12 & 0 & 16 & 0 & 19 & 0 & 23 & 0 & 25 & 0 & 28 & 1 & 30\\
$n = 15$ & 1 & 0 & 1 & 0 & 2 & 0 & 3 & 0 & 5 & 0 & 7 & 0 & 10 & 0 & 12 & 0 & 16 & 0 & 19 & 0 & 23 & 0 & 25 & 0 & 28 & 1 & 30\\
$n = 14$ & 1 & 0 & 1 & 0 & 2 & 0 & 3 & 0 & 5 & 0 & 7 & 0 & 10 & 0 & 12 & 0 & 16 & 0 & 19 & 0 & 23 & 0 & 25 & 0 & 28 & 1 & 30\\
$n = 13$ & 1 & 0 & 1 & 0 & 2 & 0 & 3 & 0 & 5 & 0 & 7 & 0 & 10 & 0 & 12 & 0 & 16 & 0 & 19 & 0 & 23 & 0 & 25 & 0 & 28 & 1 & 30\\
$n = 12$ & 1 & 0 & 1 & 0 & 2 & 0 & 3 & 0 & 5 & 0 & 7 & 0 & 10 & 0 & 12 & 0 & 16 & 0 & 19 & 0 & 23 & 0 & 25 & 0 & 28 & 1 & 29\\
$n = 11$ & 1 & 0 & 1 & 0 & 2 & 0 & 3 & 0 & 5 & 0 & 7 & 0 & 10 & 0 & 12 & 0 & 16 & 0 & 19 & 0 & 23 & 0 & 25 & 1 & 28 & 2 & 29\\
\end{tabular}
\end{table}
\end{landscape}

\begin{landscape}
\begin{table}
\caption{Betti numbers of $\Conf^n \C\P^6$ for $27 \leq i \leq 50$ and $n \leq 17$.}
\label{Tab:cp6two}
\centering
\begin{tabular}{c|c|c|c|c|c|c|c|c|c|c|c|c|c|c|c|c|c|c|c|c|c|c|c|c}
~ & $b_{27}$ & $b_{28}$ & $b_{29}$ & $b_{30}$ & $b_{31}$ & $b_{32}$ & $b_{33}$ & $b_{34}$ & $b_{35}$ & $b_{36}$ & $b_{37}$ & $b_{38}$ & $b_{39}$ & $b_{40}$ & $b_{41}$ & $b_{42}$ & $b_{43}$ & $b_{44}$ & $b_{45}$ & $b_{46}$ & $b_{47}$ & $b_{48}$ & $b_{49}$ & $b_{50}$ \\ \hline
$n = 17$ & 3 & 33 & 6 & 35 & 11 & 38 & 17 & 40 & 24 & 42 & 32 & 43 & 41 & 44 & 50 & 43 & 60 & 42 & 68 & 39 & 77 & 36 & 83 & 33\\
$n = 16$ & 3 & 33 & 6 & 35 & 11 & 38 & 17 & 39 & 24 & 41 & 32 & 41 & 41 & 41 & 50 & 39 & 60 & 37 & 68 & 33 & 77 & 30 & 83 & 27\\
$n = 15$ & 3 & 33 & 6 & 35 & 11 & 37 & 17 & 38 & 24 & 39 & 32 & 38 & 41 & 37 & 50 & 34 & 60 & 31 & 68 & 27 & 77 & 24 & 83 & 22\\
$n = 14$ & 3 & 33 & 6 & 34 & 11 & 36 & 17 & 36 & 24 & 36 & 32 & 34 & 41 & 32 & 50 & 28 & 60 & 25 & 68 & 21 & 77 & 19 & 81 & 17\\
$n = 13$ & 3 & 32 & 6 & 33 & 11 & 34 & 17 & 33 & 24 & 32 & 32 & 29 & 41 & 26 & 50 & 22 & 60 & 19 & 68 & 16 & 75 & 14 & 77 & 12\\
$n = 12$ & 3 & 31 & 6 & 31 & 11 & 31 & 17 & 29 & 24 & 27 & 32 & 23 & 41 & 20 & 50 & 16 & 60 & 14 & 66 & 11 & 71 & 9 & 70 & 7\\
$n = 11$ & 4 & 30 & 7 & 29 & 12 & 28 & 18 & 25 & 25 & 22 & 33 & 18 & 42 & 15 & 50 & 11 & 58 & 9 & 62 & 6 & 64 & 5 & 61 & 4\\
\end{tabular}
\end{table}
\end{landscape}

\begin{landscape}
\begin{table}
\caption{Betti numbers of $\Conf^n \C\P^6$ for $51 \leq i \leq 74$ and $n \leq 17$.}
\label{Tab:cp6three}
\centering
\begin{tabular}{c|c|c|c|c|c|c|c|c|c|c|c|c|c|c|c|c|c|c|c|c|c|c|c|c}
~ & $b_{51}$ & $b_{52}$ & $b_{53}$ & $b_{54}$ & $b_{55}$ & $b_{56}$ & $b_{57}$ & $b_{58}$ & $b_{59}$ & $b_{60}$ & $b_{61}$ & $b_{62}$ & $b_{63}$ & $b_{64}$ & $b_{65}$ & $b_{66}$ & $b_{67}$ & $b_{68}$ & $b_{69}$ & $b_{70}$ & $b_{71}$ & $b_{72}$ & $b_{73}$ & $b_{74}$ \\ \hline
$n = 17$ & 91 & 33 & 98 & 34 & 104 & 38 & 107 & 41 & 108 & 47 & 105 & 51 & 100 & 58 & 92 & 63 & 83 & 70 & 73 & 74 & 64 & 80 & 56 & 83\\
$n = 16$ & 91 & 28 & 96 & 29 & 100 & 33 & 100 & 36 & 98 & 42 & 92 & 46 & 85 & 53 & 75 & 58 & 66 & 65 & 56 & 70 & 49 & 75 & 41 & 76\\
$n = 15$ & 89 & 23 & 92 & 24 & 93 & 28 & 90 & 31 & 85 & 37 & 77 & 41 & 68 & 48 & 58 & 54 & 49 & 62 & 41 & 66 & 34 & 69 & 26 & 68\\
$n = 14$ & 85 & 18 & 85 & 19 & 83 & 23 & 77 & 26 & 70 & 32 & 60 & 37 & 51 & 45 & 41 & 52 & 34 & 59 & 26 & 61 & 19 & 61 & 12 & 57\\
$n = 13$ & 78 & 13 & 75 & 14 & 70 & 18 & 62 & 22 & 53 & 29 & 43 & 35 & 34 & 44 & 26 & 50 & 19 & 54 & 12 & 53 & 7 & 50 & 3 & 44\\
$n = 12$ & 68 & 8 & 62 & 10 & 55 & 15 & 45 & 20 & 36 & 28 & 26 & 35 & 19 & 42 & 12 & 45 & 7 & 46 & 3 & 42 & 1 & 37 & 0 & 29\\
$n = 11$ & 56 & 6 & 48 & 9 & 39 & 15 & 29 & 21 & 20 & 28 & 12 & 33 & 7 & 37 & 3 & 37 & 1 & 35 & 0 & 29 & 0 & 22 & 0 & 14\\
$n = 10$ & 41 & 5 & 31 & 9 & 22 & 15 & 13 & 20 & 7 & 26 & 3 & 28 & 1 & 29 & 0 & 26 & 0 & 22 & 0 & 15 & 0 & 9 & 0 & 4\\
$n = 9$ & 24 & 5 & 15 & 9 & 8 & 14 & 3 & 18 & 1 & 21 & 0 & 20 & 0 & 18 & 0 & 14 & 0 & 10 & 0 & 5 & 0 & 2 & 0 & 0\\
$n = 8$ & 10 & 5 & 4 & 8 & 1 & 12 & 0 & 13 & 0 & 13 & 0 & 10 & 0 & 8 & 0 & 5 & 0 & 3 & 0 & 1 & 0 & 0 & 0 & 0\\
$n = 7$ & 2 & 4 & 0 & 6 & 0 & 7 & 0 & 6 & 0 & 5 & 0 & 3 & 0 & 2 & 0 & 1 & 0 & 1 & 0 & 0 & 0 & 0 & 0 & 0\\
$n = 6$ & 0 & 2 & 0 & 2 & 0 & 2 & 0 & 1 & 0 & 1 & 0 & 0 & 0 & 0 & 0 & 0 & 0 & 0 & 0 & 0 & 0 & 0 & 0 & 0\\
$n \leq 5$ & 0 & 0 & 0 & 0 & 0 & 0 & 0 & 0 & 0 & 0 & 0 & 0 & 0 & 0 & 0 & 0 & 0 & 0 & 0 & 0 & 0 & 0 & 0 & 0\\
\end{tabular}
\end{table}
\end{landscape}

\begin{landscape}
\begin{table}
\caption{Betti numbers of $\Conf^n \C\P^6$ for $75 \leq i \leq 98$ and $n \leq 17$.}
\label{Tab:cp6four}
\centering
\begin{tabular}{c|c|c|c|c|c|c|c|c|c|c|c|c|c|c|c|c|c|c|c|c|c|c|c|c}
~ & $b_{75}$ & $b_{76}$ & $b_{77}$ & $b_{78}$ & $b_{79}$ & $b_{80}$ & $b_{81}$ & $b_{82}$ & $b_{83}$ & $b_{84}$ & $b_{85}$ & $b_{86}$ & $b_{87}$ & $b_{88}$ & $b_{89}$ & $b_{90}$ & $b_{91}$ & $b_{92}$ & $b_{93}$ & $b_{94}$ & $b_{95}$ & $b_{96}$ & $b_{97}$ & $b_{98}$ \\ \hline
$n = 17$ & 50 & 85 & 43 & 83 & 38 & 80 & 32 & 74 & 27 & 67 & 22 & 58 & 20 & 50 & 18 & 41 & 19 & 34 & 19 & 26 & 20 & 19 & 19 & 12\\
$n = 16$ & 35 & 76 & 28 & 72 & 23 & 67 & 18 & 59 & 15 & 51 & 13 & 42 & 14 & 34 & 15 & 26 & 17 & 19 & 17 & 12 & 17 & 7 & 15 & 3\\
$n = 15$ & 20 & 65 & 14 & 59 & 11 & 52 & 9 & 43 & 9 & 35 & 10 & 26 & 13 & 19 & 14 & 12 & 15 & 7 & 14 & 3 & 13 & 1 & 11 & 0\\
$n = 14$ & 8 & 52 & 5 & 44 & 5 & 36 & 6 & 27 & 8 & 19 & 10 & 12 & 12 & 7 & 12 & 3 & 12 & 1 & 10 & 0 & 9 & 0 & 6 & 0\\
$n = 13$ & 2 & 37 & 2 & 28 & 4 & 20 & 6 & 12 & 8 & 7 & 9 & 3 & 10 & 1 & 9 & 0 & 8 & 0 & 6 & 0 & 4 & 0 & 2 & 0\\
$n = 12$ & 1 & 21 & 2 & 13 & 4 & 7 & 6 & 3 & 7 & 1 & 7 & 0 & 7 & 0 & 5 & 0 & 4 & 0 & 2 & 0 & 1 & 0 & 0 & 0\\
$n = 11$ & 1 & 8 & 2 & 3 & 4 & 1 & 5 & 0 & 5 & 0 & 4 & 0 & 3 & 0 & 2 & 0 & 1 & 0 & 0 & 0 & 0 & 0 & 0 & 0\\
$n = 10$ & 1 & 1 & 2 & 0 & 3 & 0 & 3 & 0 & 2 & 0 & 1 & 0 & 1 & 0 & 0 & 0 & 0 & 0 & 0 & 0 & 0 & 0 & 0 & 0\\
$n = 9$ & 1 & 0 & 1 & 0 & 1 & 0 & 1 & 0 & 0 & 0 & 0 & 0 & 0 & 0 & 0 & 0 & 0 & 0 & 0 & 0 & 0 & 0 & 0 & 0\\
$n \leq 8$ & 0 & 0 & 0 & 0 & 0 & 0 & 0 & 0 & 0 & 0 & 0 & 0 & 0 & 0 & 0 & 0 & 0 & 0 & 0 & 0 & 0 & 0 & 0 & 0
\end{tabular}
\end{table}
\end{landscape}

\begin{landscape}
\begin{table}
\caption{Betti numbers of $\Conf^n \C\P^6$ for $99 \leq i \leq 115$ and $n \leq 17$.}
\label{Tab:cp6five}
\centering
\begin{tabular}{c|c|c|c|c|c|c|c|c|c|c|c|c|c|c|c|c|c}
~ & $b_{99}$ & $b_{100}$ & $b_{101}$ & $b_{102}$ & $b_{103}$ & $b_{104}$ & $b_{105}$ & $b_{106}$ & $b_{107}$ & $b_{108}$ & $b_{109}$ & $b_{110}$ & $b_{111}$ & $b_{112}$ & $b_{113}$ & $b_{114}$ & $b_{115}$ \\ 
\hline
$n = 17$ & 18 & 7 & 16 & 3 & 14 & 1 & 11 & 0 & 9 & 0 & 6 & 0 & 4 & 0 & 2 & 0 & 1\\
$n = 16$ & 14 & 1 & 11 & 0 & 9 & 0 & 6 & 0 & 4 & 0 & 2 & 0 & 1 & 0 & 0 & 0 & 0\\
$n = 15$ & 9 & 0 & 6 & 0 & 4 & 0 & 2 & 0 & 1 & 0 & 0 & 0 & 0 & 0 & 0 & 0 & 0\\
$n = 14$ & 4 & 0 & 2 & 0 & 1 & 0 & 0 & 0 & 0 & 0 & 0 & 0 & 0 & 0 & 0 & 0 & 0\\
$n = 13$ & 1 & 0 & 0 & 0 & 0 & 0 & 0 & 0 & 0 & 0 & 0 & 0 & 0 & 0 & 0 & 0 & 0\\
$n \leq 12$ & 0 & 0 & 0 & 0 & 0 & 0 & 0 & 0 & 0 & 0 & 0 & 0 & 0 & 0 & 0 & 0 & 0\\
\end{tabular}
\end{table}
\end{landscape}

\begin{landscape}
\begin{table}
\caption{Betti numbers of $\Conf^n \Sigma_2$ for $i \leq 21$ and $n \leq 21$.}
\label{Tab:g2one}
\centering
\begin{tabular}{c|c|c|c|c|c|c|c|c|c|c|c|c|c|c|c|c|c|c|c|c|c|c}
~ & $b_0$ & $b_1$ & $b_2$ & $b_3$ & $b_4$ & $b_5$ & $b_6$ & $b_7$ & $b_8$ & $b_9$ & $b_{10}$ & $b_{11}$ & $b_{12}$ & $b_{13}$ & $b_{14}$ & $b_{15}$ & $b_{16}$ & $b_{17}$ & $b_{18}$ & $b_{19}$ & $b_{20}$ & $b_{21}$ \\ 
\hline
$n = 21$ & 1 & 4 & 6 & 16 & 28 & 48 & 75 & 114 & 162 & 225 & 300 & 393 & 501 & 630 & 777 & 948 & 1140 & 1359 & 1602 & 1875 & 2175 & 1783\\
$n = 20$ & 1 & 4 & 6 & 16 & 28 & 48 & 75 & 114 & 162 & 225 & 300 & 393 & 501 & 630 & 777 & 948 & 1140 & 1359 & 1602 & 1875 & 1635 & 594\\
$n = 19$ & 1 & 4 & 6 & 16 & 28 & 48 & 75 & 114 & 162 & 225 & 300 & 393 & 501 & 630 & 777 & 948 & 1140 & 1359 & 1602 & 1326 & 451 & 0\\
$n = 18$ & 1 & 4 & 6 & 16 & 28 & 48 & 75 & 114 & 162 & 225 & 300 & 393 & 501 & 630 & 777 & 948 & 1140 & 1359 & 1206 & 440 & 0 & 0\\
$n = 17$ & 1 & 4 & 6 & 16 & 28 & 48 & 75 & 114 & 162 & 225 & 300 & 393 & 501 & 630 & 777 & 948 & 1140 & 955 & 325 & 0 & 0 & 0\\
$n = 16$ & 1 & 4 & 6 & 16 & 28 & 48 & 75 & 114 & 162 & 225 & 300 & 393 & 501 & 630 & 777 & 948 & 860 & 315 & 0 & 0 & 0 & 0\\
$n = 15$ & 1 & 4 & 6 & 16 & 28 & 48 & 75 & 114 & 162 & 225 & 300 & 393 & 501 & 630 & 777 & 661 & 225 & 0 & 0 & 0 & 0 & 0\\
\end{tabular}
\end{table}
\end{landscape}

\begin{landscape}
\begin{table}
\caption{Betti numbers of $\Conf^n \Sigma_3$ for $i \leq 16$ and $n \leq 15$.}
\label{Tab:g3one}
\centering
\begin{tabular}{c|c|c|c|c|c|c|c|c|c|c|c|c|c|c|c|c|c}
~ & $b_0$ & $b_1$ & $b_2$ & $b_3$ & $b_4$ & $b_5$ & $b_6$ & $b_7$ & $b_8$ & $b_9$ & $b_{10}$ & $b_{11}$ & $b_{12}$ & $b_{13}$ & $b_{14}$ & $b_{15}$ & $b_{16}$ \\ 
\hline
$n = 15$ & 1 & 6 & 15 & 36 & 90 & 169 & 335 & 569 & 979 & 1531 & 2396 & 3520 & 5151 & 7211 & 10039 & 10575 & 3795\\
$n = 14$ & 1 & 6 & 15 & 36 & 90 & 169 & 335 & 569 & 979 & 1531 & 2396 & 3520 & 5151 & 7211 & 7260 & 2505 & 0\\
$n = 13$ & 1 & 6 & 15 & 36 & 90 & 169 & 335 & 569 & 979 & 1531 & 2396 & 3520 & 5151 & 5706 & 2010 & 0 & 0\\
$n = 12$ & 1 & 6 & 15 & 36 & 90 & 169 & 335 & 569 & 979 & 1531 & 2396 & 3520 & 3736 & 1266 & 0 & 0 & 0\\
$n = 11$ & 1 & 6 & 15 & 36 & 90 & 169 & 335 & 569 & 979 & 1531 & 2396 & 2835 & 966 & 0 & 0 & 0 & 0\\
$n = 10$ & 1 & 6 & 15 & 36 & 90 & 169 & 335 & 569 & 979 & 1531 & 1751 & 574 & 0 & 0 & 0 & 0 & 0\\
$n = 9$ & 1 & 6 & 15 & 36 & 90 & 169 & 335 & 569 & 979 & 1266 & 406 & 0 & 0 & 0 & 0 & 0 & 0\\
$n = 8$ & 1 & 6 & 15 & 36 & 90 & 169 & 335 & 569 & 728 & 224 & 0 & 0 & 0 & 0 & 0 & 0 & 0\\
$n = 7$ & 1 & 6 & 15 & 36 & 90 & 169 & 335 & 490 & 140 & 0 & 0 & 0 & 0 & 0 & 0 & 0 & 0\\
$n = 6$ & 1 & 6 & 15 & 36 & 90 & 169 & 259 & 70 & 0 & 0 & 0 & 0 & 0 & 0 & 0 & 0 & 0\\
$n = 5$ & 1 & 6 & 15 & 36 & 90 & 155 & 35 & 0 & 0 & 0 & 0 & 0 & 0 & 0 & 0 & 0 & 0\\
$n = 4$ & 1 & 6 & 15 & 36 & 76 & 15 & 0 & 0 & 0 & 0 & 0 & 0 & 0 & 0 & 0 & 0 & 0\\
$n = 3$ & 1 & 6 & 15 & 36 & 6 & 0 & 0 & 0 & 0 & 0 & 0 & 0 & 0 & 0 & 0 & 0 & 0\\
$n = 2$ & 1 & 6 & 15 & 0 & 0 & 0 & 0 & 0 & 0 & 0 & 0 & 0 & 0 & 0 & 0 & 0 & 0\\
$n = 1$ & 1 & 6 & 1 & 0 & 0 & 0 & 0 & 0 & 0 & 0 & 0 & 0 & 0 & 0 & 0 & 0 & 0
\end{tabular}
\end{table}
\end{landscape}

\begin{landscape}
\begin{table}
\caption{Betti numbers of $\Conf^n \Sigma_4$ for $i \leq 11$ and $n \leq 11$.}
\label{Tab:g4one}
\centering
\begin{tabular}{c|c|c|c|c|c|c|c|c|c|c|c|c}
~ & $b_0$ & $b_1$ & $b_2$ & $b_3$ & $b_4$ & $b_5$ & $b_6$ & $b_7$ & $b_8$ & $b_9$ & $b_{10}$ & $b_{11}$\\ 
\hline
$n = 11$ & 1 & 8 & 28 & 85 & 218 & 532 & 1098 & 2289 & 4187 & 7748 & 11451 & 3318\\
$n = 10$ & 1 & 8 & 28 & 85 & 218 & 532 & 1098 & 2289 & 4187 & 6384 & 1764 & 0\\
$n = 9$ & 1 & 8 & 28 & 85 & 218 & 532 & 1098 & 2289 & 3822 & 966 & 0 & 0\\
$n = 8$ & 1 & 8 & 28 & 85 & 218 & 532 & 1098 & 1974 & 462 & 0 & 0 & 0\\
$n = 7$ & 1 & 8 & 28 & 85 & 218 & 532 & 1050 & 210 & 0 & 0 & 0 & 0\\
$n = 6$ & 1 & 8 & 28 & 85 & 218 & 490 & 84 & 0 & 0 & 0 & 0 & 0\\
$n = 5$ & 1 & 8 & 28 & 85 & 218 & 28 & 0 & 0 & 0 & 0 & 0 & 0\\
$n = 4$ & 1 & 8 & 28 & 85 & 8 & 0 & 0 & 0 & 0 & 0 & 0 & 0\\
$n = 3$ & 1 & 8 & 28 & 0 & 0 & 0 & 0 & 0 & 0 & 0 & 0 & 0\\
$n = 2$ & 1 & 8 & 1 & 0 & 0 & 0 & 0 & 0 & 0 & 0 & 0 & 0\\
$n = 1$ & 1 & 0 & 0 & 0 & 0 & 0 & 0 & 0 & 0 & 0 & 0 & 0 
\end{tabular}
\end{table}
\end{landscape}

\begin{landscape}
\begin{table}
\caption{Betti numbers of $\Conf^n \Sigma_1 \times \C\P^1$ for $i \leq 23$ and $n \leq 15$.}
\label{Tab:g1xcp1one}
\centering
\begin{tabular}{c|c|c|c|c|c|c|c|c|c|c|c|c|c|c|c|c|c|c|c|c|c|c|c|c}
~ & $b_0$ & $b_1$ & $b_2$ & $b_3$ & $b_4$ & $b_5$ & $b_6$ & $b_7$ & $b_8$ & $b_9$ & $b_{10}$ & $b_{11}$ & $b_{12}$ & $b_{13}$ & $b_{14}$ & $b_{15}$ & $b_{16}$ & $b_{17}$ & $b_{18}$ & $b_{19}$ & $b_{20}$ & $b_{21}$ & $b_{22}$ & $b_{23}$\\ \hline
$n = 15$ & 1 & 2 & 3 & 6 & 9 & 10 & 12 & 19 & 30 & 42 & 54 & 67 & 83 & 104 & 134 & 175 & 219 & 259 & 304 & 367 & 448 & 539 & 636 & 741\\
$n = 14$ & 1 & 2 & 3 & 6 & 9 & 10 & 12 & 19 & 30 & 42 & 54 & 67 & 83 & 104 & 134 & 175 & 219 & 259 & 304 & 367 & 448 & 539 & 636 & 741\\
$n = 13$ & 1 & 2 & 3 & 6 & 9 & 10 & 12 & 19 & 30 & 42 & 54 & 67 & 83 & 104 & 134 & 175 & 219 & 259 & 304 & 367 & 448 & 539 & 636 & 741\\
$n = 12$ & 1 & 2 & 3 & 6 & 9 & 10 & 12 & 19 & 30 & 42 & 54 & 67 & 83 & 104 & 134 & 175 & 219 & 259 & 304 & 367 & 448 & 539 & 636 & 741\\
$n = 11$ & 1 & 2 & 3 & 6 & 9 & 10 & 12 & 19 & 30 & 42 & 54 & 67 & 83 & 104 & 134 & 175 & 219 & 259 & 304 & 367 & 448 & 539 & 620 & 661\\
$n = 10$ & 1 & 2 & 3 & 6 & 9 & 10 & 12 & 19 & 30 & 42 & 54 & 67 & 83 & 104 & 134 & 175 & 219 & 259 & 304 & 367 & 432 & 470 & 476 & 447\\
$n = 9$ & 1 & 2 & 3 & 6 & 9 & 10 & 12 & 19 & 30 & 42 & 54 & 67 & 83 & 104 & 134 & 175 & 219 & 259 & 293 & 312 & 314 & 300 & 260 & 196\\
\end{tabular}
\end{table}
\end{landscape}

\begin{landscape}
\begin{table}
\caption{Betti numbers of $\Conf^n \Sigma_1 \times \C\P^1$ for $24 \leq i \leq 46$ and $n \leq 15$.}
\label{Tab:g1xcp1two}
\centering
\begin{tabular}{c|c|c|c|c|c|c|c|c|c|c|c|c|c|c|c|c|c|c}
~ & $b_{24}$ & $b_{25}$ & $b_{26}$ & $b_{27}$ & $b_{28}$ & $b_{29}$ & $b_{30}$ & $b_{31}$ & $b_{32}$ & $b_{33}$ & $b_{34}$ & $b_{35}$ & $b_{36}$ & $b_{37}$ & $b_{38}$ & $b_{39}$ & $b_{40}$ & $b_{41}$\\ \hline
$n = 15$ & 857 & 989 & 1148 & 1337 & 1539 & 1742 & 1933 & 2076 & 2158 & 2172 & 2084 & 1894 & 1643 & 1373 & 1114 & 876 & 653 & 445\\
$n = 14$ & 857 & 989 & 1148 & 1337 & 1510 & 1612 & 1651 & 1637 & 1556 & 1413 & 1220 & 993 & 761 & 554 & 390 & 267 & 167 & 85\\
$n = 13$ & 857 & 989 & 1126 & 1227 & 1266 & 1239 & 1144 & 1002 & 844 &  684 & 522 & 362 & 221 & 117 & 56 & 25 & 7 & 0\\
$n = 12$ & 835 & 892 & 919 & 911 & 843 & 720 & 572 & 429 & 310 & 217 & 138 & 71 & 26 & 5 & 0 & 0 & 0 & 0\\
$n = 11$ & 660 & 631 & 574 & 492 & 391 & 279 & 174 & 94 & 46 & 21 & 6 & 0 & 0 & 0 & 0 & 0 & 0 & 0\\
$n = 10$ & 383 & 304 & 230 & 167 & 109 & 57 & 21 & 4 & 0 & 0 & 0 & 0 & 0 & 0 & 0 & 0 & 0 & 0\\
$n = 9$ & 127 & 71 & 36 & 17 & 5 & 0 & 0 & 0 & 0 & 0 & 0 & 0 & 0 & 0 & 0 & 0 & 0 & 0\\
$n = 8$ & 16 & 3 & 0 & 0 & 0 & 0 & 0 & 0 & 0 & 0 & 0 & 0 & 0 & 0 & 0 & 0 & 0 & 0\\
$n \leq 7$ & 0 & 0 & 0 & 0 & 0 & 0 & 0 & 0 & 0 & 0 & 0 & 0 & 0 & 0 & 0 & 0 & 0 & 0\\
\end{tabular}
\end{table}
\end{landscape}

\subsection*{Acknowledgements} 

This work was partially supported by National Science Foundation grants DMS-1147782, DMS-1301690, DMS-0838210, and DMS-1502553 and the Alfred P. Sloan foundation. The author would like to thank Melanie Matchett Wood for her insights and vast editing help. The author would also like to thank Ben Knudsen, Jeremy Miller, Dan Petersen, and Craig Westerland for helpful conversation.

\end{document}